\documentclass[12pt,reqno]{amsart}
\usepackage[dvipsnames]{xcolor}
\usepackage{graphicx}
\usepackage{amsmath, amssymb}
\usepackage{mathrsfs}
\usepackage{amsthm}
\usepackage{color}
\usepackage{mathtools}
\usepackage{nameref}
\usepackage{url}
\usepackage[colorlinks=true,linkcolor=blue!60!black,citecolor=teal!60!black,urlcolor=RoyalBlue]{hyperref}
\usepackage[utf8x]{inputenc} 
\usepackage[left=2.8cm,right=2.8cm,top=2.8cm,bottom=2.8cm]{geometry}
\usepackage{enumitem,kantlipsum}
\usepackage{accents}
\usepackage{import}

\newtheorem{thm}{Theorem}[section]
\newtheorem{cor}[thm]{Corollary}
\newtheorem{lemma}[thm]{Lemma}

\newtheorem{obs}[thm]{Observation}
\newtheorem{question}[thm]{Question}
\newtheorem{dfn&prop}[thm]{Definition and Proposition}
\newtheorem{dfn&thm}[thm]{Definition and Theorem}
\newtheorem*{Hormander}{Hörmander's Theorem}

\theoremstyle{definition}

\newtheorem{observation}[thm]{Observation}

\theoremstyle{remark}
\newtheorem*{Claim}{Claim}

\newtheorem*{remark}{Remark}

\numberwithin{equation}{section}

\newenvironment{subproof}{\begin{proof}[Proof of claim.]}{%
\end{proof}}
\newenvironment{subproof*}{\begin{proof}[Sketch of proof.]}{%
\end{proof}}

\def\phi{\varphi}

\def\C{{\mathbb{C}}}
\def\D{{\mathbb{D}}}
\def\N{{\mathbb{N}}}

\def\R{{\mathbb{R}}}

\newcommand{\dist}{\operatorname{dist}}

\newcommand{\Ima}{\operatorname{Im}}
\newcommand{\Rea}{\operatorname{Re}}

\newcommand{\Arg}{\operatorname{Arg}}

\newcommand \eps {\varepsilon}

\newcommand \nin{\not\in}

\newcommand* \bb [1]{\left({#1}\right)}
\newcommand* \sbb [1]{\left[{#1}\right]} 

\newcommand* \bset[1] {\left\{{#1}\right\}}

\newcommand* \integrate[4]{\int_{#1}^{#2}{#3}d{#4}} 
\newcommand* \limit [2]{\underset{{#1}\rightarrow{#2}}{\lim}\;}
\newcommand* \bunion[3]{\bigcup_{{#1} = {#2}}^{#3}}

\newcommand* \sumit [3]{\underset{{#1} = {#2}}{\overset{#3}\sum}\;}
\newcommand* \prodit [3]{\underset{{#1} = {#2}}{\overset{#3}\prod}\;}

\newcommand* \abs[1] {\left|{#1}\right|}

\newcommand*{\defeq}{\mathrel{\vcenter{\baselineskip0.5ex \lineskiplimit0pt
                                    \hbox{\scriptsize.}\hbox{\scriptsize.}}}%
            =}

\author[{V. Evdoridou \and A. Gl\"ucksam \and L. Pardo-Sim\'on}]{Vasiliki Evdoridou \and Adi Gl\"ucksam \and Leticia Pardo-Sim\'on}
\address{\noindent School of Mathematics and Statistics \\The Open University\\ Walton Hall\\ Milton Keynes MK7 6AA\\
	UK \\ \textsc{\newline \indent \href{https://orcid.org/0000-0002-5409-2663}{\includegraphics[width=1em,height=1em]{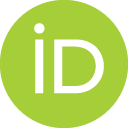} {\normalfont https://orcid.org/0000-0002-5409-2663}}}}
\email{vasiliki.evdoridou@open.ac.uk}
\address{\noindent Dept. of Mathematics \\ Northwestern University \\ Evanston \\ IL 60208 \\ US \\ 
	\textsc{\newline \indent 
		\href{https://orcid.org/0000-0002-6957-9431%
		}{\includegraphics[width=1em,height=1em]{orcid2.png} {\normalfont https://orcid.org/0000-0002-6957-9431}}
}}
\email{adiglucksam@gmail.com}
\address{\noindent Dept. of Mathematics \\ The University of Manchester \\ Manchester \\ M13 9PL \\ UK\\ 
	\textsc{\newline \indent 
		\href{https://orcid.org/0000-0003-4039-5556%
		}{\includegraphics[width=1em,height=1em]{orcid2.png} {\normalfont https://orcid.org/0000-0003-4039-5556}}
}}
\email{leticia.pardosimon@manchester.ac.uk}


\begin{document}

\title[Unbounded wandering domains]{Unbounded fast escaping wandering domains}

\maketitle
\begin{abstract}
We introduce a new approximation technique into the context of complex dynamics that allows us to construct examples of transcendental entire functions with unbounded wandering domains. We provide examples of entire functions with an orbit of unbounded \textit{fast escaping} wandering domains, answering a long-standing question of Rippon and Stallard. Moreover, these examples cover all possible types of simply-connected wandering domains in terms of \textit{convergence to the boundary}. In relation to a conjecture of Baker, it was unknown whether functions of order less than one could have unbounded wandering domains. For any given order greater than $1/2$ and smaller than~$1$, we provide an entire function of such order with an unbounded wandering domain. 
\end{abstract}
\section{Introduction}

Let $f\colon\C\to\C$ be a transcendental entire function and for every $n\geq 0$, denote by $f^n$ its $n$-th iterate.  The  \textit{Fatou set} of $f$, $F(f)$, consists of all the points $z\in\C$ for which the family $\{f^n\}_{n\in \N}$ is equicontinuous with respect to the spherical metric. Its complement, $J(f) \defeq \C\setminus F(f)$, is known as the \emph{Julia set}. The Fatou set of an entire map is always open and consists of connected components, called \textit{Fatou components}. For a Fatou component $U$ of $f$ and each $n\geq 0$, we denote by $U_n$ the Fatou component containing $f^n(U)$, and call the sequence $(U_n)_{n\geq 0}$ an \textit{orbit}. A Fatou component $U$ may be \textit{periodic}, if $U_n \subseteq U$ for some $n\in \mathbb{N},$ \textit{preperiodic}, if $U_n$ is periodic for some $n \in \mathbb{N},$ or, otherwise, $U$ is called a \textit{wandering domain}. Equivalently, $U$ is a wandering domain if $U_n=U_m$ implies $n=m$. 

Baker \cite{Baker_wandering_76} was the first to provide an example of a transcendental entire function with a  wandering domain in 1976, while Sullivan \cite{Sullivan_noWD_85} showed in 1985 that they do not occur for rational functions. Since then, many examples of transcendental entire functions with wandering domains, most of them bounded, have been provided, e.g. \cite{Baker_wandering84, Herman_84,Erem_Lyubich_pathological_87, Devaney_90, fagella_henriksen_Teich_09}. In recent years, with the development of new techniques, much progress on the understanding of  wandering domains has been achieved, e.g.  \cite{bishop2015, Lazebnik_several_17, david_Shishikura_Wandering, Fagella_Kirill_19,BEFRS_internal22,BocThaler_21,  lasse_DJ_Erem_22}.

All limit functions of the family of iterates in a wandering domain are constant \cite{Fatou_20}, and so they can be classified into \textit{escaping}, when the only limit function is infinity, \textit{oscillating}, if both infinity and at least one other finite value are limit functions, and  \textit{dynamically bounded}, when all limit functions are in $\C$. It remains one of the major open problems in transcendental dynamics whether dynamically bounded wandering domains exist. In this paper, we focus on escaping wandering domains.

For a transcendental entire function $f$, points might converge to infinity under iteration at different rates. Of particular importance is the \textit{fast escaping set}, $A(f)$, of points  that escape to infinity ‘as fast as possible’ under iteration. This set, introduced by Bergweiler and Hinkkanen \cite{Bergweiler_Hinkkanen_fast99} in connection to permutable functions and absence of wandering domains, is defined in \cite{PG_fastescaping_12} as
\begin{equation*}
	A(f) \defeq \{z\colon  \text{there exists } \ell \in \N \text{ such that } \vert f^{n+\ell}(z) \vert\geq M_f^n(R), \text{ for } n\in \N\},
\end{equation*}
where $M_f(r) \defeq \max_{\vert z \vert=r}|f(z)|$, for $r >0$, is the maximum modulus function, $M_f^n(r)$ denotes the $n$-th iterate of the function $M_f(r)$, and $R >0$ can be taken to be any value such that $M_f(r)> r$, for $r\geq R$.  It follows from \cite[Theorem 1.2]{PG_fastescaping_12} that if $U$ is a component of $F(f)$ such that $U\cap A(f)\neq \emptyset$, then $\overline{U}\subset A(f)$, and we say that $U$ is \textit{fast escaping}. We note that all fast escaping components of $F(f)$ must be wandering domains; see \cite[Lemma 4]{Bergweiler_Hinkkanen_fast99}.

Multiply connected wandering domains are always bounded and fast escaping, \cite{Baker_wandering84, PG_questions05}. However, only Bergweiler \cite{bergweiler_fast11} and Sixsmith \cite{dave_fast12} have provided examples of simply connected, fast escaping wandering domains, which are shown to be bounded in both cases. This led Rippon and Stallard to ask whether there exist unbounded fast escaping wandering domains, \cite[Question 1, p.~802]{PG_fastescaping_12}; see also \cite[p.~453]{Benini_permutable16} for the connection of this problem to that of permutable functions having the same Julia set. We give a positive answer to Rippon and Stallard's question. 

\begin{thm}\label{thm_fast_escaping}
	There exists a transcendental entire function with an orbit of unbounded fast escaping wandering domains.
\end{thm}

We prove Theorem \ref{thm_fast_escaping} by introducing a new approximation technique into the world of complex dynamics allowing us to construct functions with unbounded wandering domains. Our method is based on Hörmander’s solution to $\bar\partial$-equations; see \hyperref[thm:Hormander]{Hörmander's Theorem} on p.9. In particular, it provides control over the growth of the resulting function while allowing for flexibility on the choice of domains of approximation, see Section \ref{sec:pasting} for details. 

While all orbits \textit{within} (pre)periodic Fatou components and multiply connected wandering domains behave essentially in the same manner, see \cite{Bergweilermerophormic, BergweilerPG_multiply13}, the internal dynamics of simply connected wandering domains exhibit a wider range of possibilities, and have only been completely classified recently. Namely, in \cite{BEFRS_internal22}, two classifications are provided, giving rise to three cases each. The first in terms of the long-term behaviour of orbits with respect to \textit{hyperbolic distances}, \cite[Theorem A]{BEFRS_internal22}, and the second with respect to \textit{convergence to the boundary}, \cite[Theorem C]{BEFRS_internal22}; for the definition, see Theorem~\ref{thm_internal}.  Examples of bounded escaping wandering domains of each possible type are provided in \cite{BEFRS_internal22}; and of bounded oscillating ones in \cite{evdoridou_rippon_stallard_2022}. This prompts the question of whether there are examples of unbounded wandering domains with all different possible internal dynamics.  We provide examples with all the different behaviours in terms of convergence to the boundary, showing that the function from Theorem \ref{thm_fast_escaping} can be chosen to have wandering domains of any of the types arising from \cite[Theorem~C]{BEFRS_internal22}.

\begin{thm}\label{thm_cases_fast}
For each of the three possible types of simply connected wandering domains in terms of convergence to the boundary, there exists a transcendental entire function with an orbit of unbounded fast escaping wandering domains of that type.
\end{thm}

\begin{remark}
In Observation \ref{obs:contracting} we explain how we can choose the above wandering domains to be `contracting' with respect to hyperbolic distances, and discuss the other two types arising from \cite[Theorem A]{BEFRS_internal22}.
	\end{remark}

The existence of unbounded wandering domains is related to a conjecture of Baker from 1981, \cite{baker_conjecture_81}.  Part of this conjecture states that a transcendental entire function $f$ of order less than half cannot have unbounded Fatou components. Recall that the \textit{order} of an entire map $f$ is defined as $\rho(f)\defeq \limsup_{r\to \infty}\frac{\log \log M_f(r)}{\log r}$. Zheng \cite{Zheng_unbounded_00} showed that such functions have no unbounded periodic or preperiodic Fatou components, and it is still an open question if such a map can have an unbounded wandering domain. In addition, the conjecture has been shown to hold in many instances where growth conditions are imposed on the function, see e.g. \cite{Hinkkanen_Miles_growth_09, PG_small_growth_09, PG_Baker_Erem_13, Nicks_PG_Baker18} and the survey \cite{Hinkkanen_survey_Baker_09}. 

Absence of unbounded wandering domains has also been shown for further functions of small order. Recently, Nicks, Rippon and Stallard \cite{Nicks_PG_Baker18} showed that real entire functions with only real zeros and of order less than one do not have orbits of unbounded wandering domains. Note however that there are known examples of functions of order one with orbits of unbounded wandering domains, as shown by Herman for the map $z \mapsto z-1+e^{-z}+2\pi i $, \cite{Herman_84}. This suggests the following question.
\begin{question}[{\cite[Problem 2.93, p. 61]{Hayman_book_50}, \cite[p. 101]{Nicks_PG_Baker18}}] Let $f$ be a transcendental entire function of order less than one. Can $f$ have (an orbit of) unbounded wandering domains?
\end{question}

The following theorem answers the question by providing the first examples of entire functions of order less than $1$ with unbounded wandering domains. 

\begin{thm}\label{thm_order1/2}
	For every $\eps\in (0,1/2]$, there exists an entire function of order $1/2+\eps$ with an unbounded fast escaping wandering domain.
\end{thm}
 \begin{remark}
The functions constructed in Theorems \ref{thm_fast_escaping} and \ref{thm_cases_fast} are of infinite order of growth, see \eqref{eq_tauk_fast} in Corollary \ref{cor_wd_pasting}. Therefore, new tools are required to prove Theorem~\ref{thm_order1/2}. We note that our construction ensures that the resulting function has one unbounded wandering domain, $U_0$. However, we were not able to show whether any element, $U_n$, in its orbit is bounded or unbounded; see Corollary \ref{cor:sequence} for details. It remains an open question whether there exists an entire function of order less than $1$ with an orbit of unbounded wandering domains.
\end{remark}

\subsection*{Notation} Given a set $U\subset\C$ and $\delta>0$,  we let
$$
U^{-\delta}\defeq\bset{z\in U \colon \dist(z,\C\setminus U)>\delta}
\quad \text{ and } \quad U^{+\delta}\defeq\bset{z \in \C \colon \dist(z, U)<\delta},$$
i.e., the set $U^{-\delta}$ is the subset of $U$ where all the points are separated from the boundary of $U$ by at least $\delta$ and the set $U^{+\delta}$ is a `fattening' or `padding' of the set $U$ by $\delta$.

Given a sequence $\bset{a_n}$ we will use the notation $a_n\searrow 0$ to indicate that $\bset{a_n}$ is monotone decreasing to $0$ and the notation $a_n\nearrow\infty$ to indicate that $\bset{a_n}$ is monotone increasing to $\infty$. We will denote by $S$ the strip of width one centred about the real line, i.e., $$S\defeq \{z\in \C \colon \abs{\Ima(z)}<1/2\},$$ and by $dm$ Lebesgue's measure on the complex plane. We denote by $B(z,r)$ the disk centred at $z$ of radius $r>0$ and by $n B(z,r), n \geq 2$, the disk centred at $z$ of radius $n\cdot r$. Given a set $T\subset \C$ and a constant $a\in \C$, we denote by $T+a$ the translation of $T$ by $a$, i.e., $T+a\defeq\{z+a \colon z\in T\}$. The Euclidean distance is denoted by $\dist$, the hyperbolic distance in a simply connected domain $U$ is denoted by $\dist_U$ and the symbol $\triangle$ denotes the end of the proof of a claim within the proof of a theorem. By \textit{numerical constant} we mean a constant number that does not depend on any parameter.

\subsection*{Acknowledgements}
We thank Dave Sixsmith for helpful comments and the referee for many comments and suggestions which greatly improved the presentation of the paper. This material is partially based upon work supported by the National Science Foundation under Grant No. DMS-1928930 while the authors participated in programs hosted by the Mathematical Sciences Research Institute in Berkeley, California, during the Spring 2022 semester.

\section{Orbits of unbounded wandering domains}\label{sec:pasting}

Our main goal in this section is to prove the following flexible construction theorem, from which Theorems \ref{thm_fast_escaping} and \ref{thm_cases_fast} will follow readily. Roughly speaking, it tells us that given a collection of holomorphic maps, $h_k$, defined on the horizontal strip $S$, there exists an entire function $f$ that approximates, for each $k \geq 1$, a translation of $h_k$ from $S^{-\eps_k}+i\tau_k$ to $S+i\tau_{k+1}$ for sequences $\eps_k \searrow 0$ and $\tau_k \nearrow \infty$, up to a prescribed error, $\delta_k$. Moreover, $f$ will map a strip of points in between each copy of $S$ close to the point $i\tau_0$; see Figure~\ref{fig:idea_pasting}. In particular, we will show in Corollary \ref{cor_wd_pasting} that $f$ has an orbit, $\{U_n\}$, of unbounded wandering domains, asymptotically equal to horizontal strips and separated by preimages of a basin of attraction containing $i\tau_0$, that are fast escaping whenever the translation parameters, $\tau_k$, diverge to infinity sufficiently fast. 

\begin{thm} \label{thm:pasting_strips}
Given 
\begin{enumerate}[label=(\roman*)]
	\item A sequence $\bset{\eps_k}\subset\bb{0,\frac14}$ with $\eps_k\searrow 0$;
	\item A sequence $\tau_k\nearrow\infty$ with $\tau_0\ge\frac32$, $\tau_{k+1}>\tau_{k}+1+\eps_{k}+\eps_{k+1}$ and such that 
	\begin{equation}\label{eq_alphak}
	\delta_k\defeq \frac3{\eps_k\cdot \tau_k^2}<\eps_{k+1}\to 0;
	\end{equation}
	\item The map $h_0\colon S\to \C$ defined as $h_0(z)\defeq i\bb{\tau_0-\tau_1}$  and a sequence $\{h_k \}_{k \geq 1}$ of holomorphic maps, $h_k\colon S\rightarrow S$,  such that 
	\begin{equation*}
	\abs{h_k(z)}\le\abs z^{n_k}+b_k,
	\end{equation*}
	 where $\bset{b_k}, \bset{n_k}\subset\N$ are  non-decreasing sequences;
\end{enumerate}
There exists a monotone increasing sequence $\bset{a_k}$, where $a_k$ depends only on $\{h_j\}^{k+1}_{j=0}$, $\{\eps_j\}^{k+1}_{j=0}$ and $\{\tau_{j}\}^{k+1}_{j=0}$, and an entire function $f\colon \C\rightarrow\C$ with the following properties.
\begin{enumerate}
	\item \label{item:pasting_f_hk} For every $k\ge 0$, for every $z\in S^{-\eps_k}+i\tau_k$,
	$$
	\abs{f(z)-\bb{h_k(z-i\tau_k)+i\tau_{k+1}}}< \delta_k.
	$$
	\item   \label{item:pasting_f_i}For every $k\ge 0$ and $z\in W_k\defeq \bset{w\colon \tau_k+1/2+\eps_k< \Ima(w)< \tau_{k+1}-1/2-\eps_{k+1}}$,
	$$
	\abs{f(z)-i\tau_0}<\delta_k.
	$$
	\item   \label{item:pasting_upperbound} If $\Ima(z)\le\tau_{k+1}-\frac32$, then 
	$$
	\abs{f(z)}\le 4\exp\bb{a_k\exp\bb{\frac\pi{\eps_k}\abs{\Rea(z)}}}.
	$$
\end{enumerate}
\end{thm}

\begin{figure}[htp]
	\centering
	\def\svgwidth{\linewidth}
\begingroup%
\makeatletter%
\providecommand\color[2][]{%
	\errmessage{(Inkscape) Color is used for the text in Inkscape, but the package 'color.sty' is not loaded}%
	\renewcommand\color[2][]{}%
}%
\providecommand\transparent[1]{%
	\errmessage{(Inkscape) Transparency is used (non-zero) for the text in Inkscape, but the package 'transparent.sty' is not loaded}%
	\renewcommand\transparent[1]{}%
}%
\providecommand\rotatebox[2]{#2}%
\newcommand*\fsize{\dimexpr\f@size pt\relax}%
\newcommand*\lineheight[1]{\fontsize{\fsize}{#1\fsize}\selectfont}%
\ifx\svgwidth\undefined%
\setlength{\unitlength}{841.88976378bp}%
\ifx\svgscale\undefined%
\relax%
\else%
\setlength{\unitlength}{\unitlength * \real{\svgscale}}%
\fi%
\else%
\setlength{\unitlength}{\svgwidth}%
\fi%
\global\let\svgwidth\undefined%
\global\let\svgscale\undefined%
\makeatother%
\begin{picture}(1,0.51515152)%
	\lineheight{1}%
	\setlength\tabcolsep{0pt}%
	\put(0,0){\includegraphics[width=\unitlength,page=1]{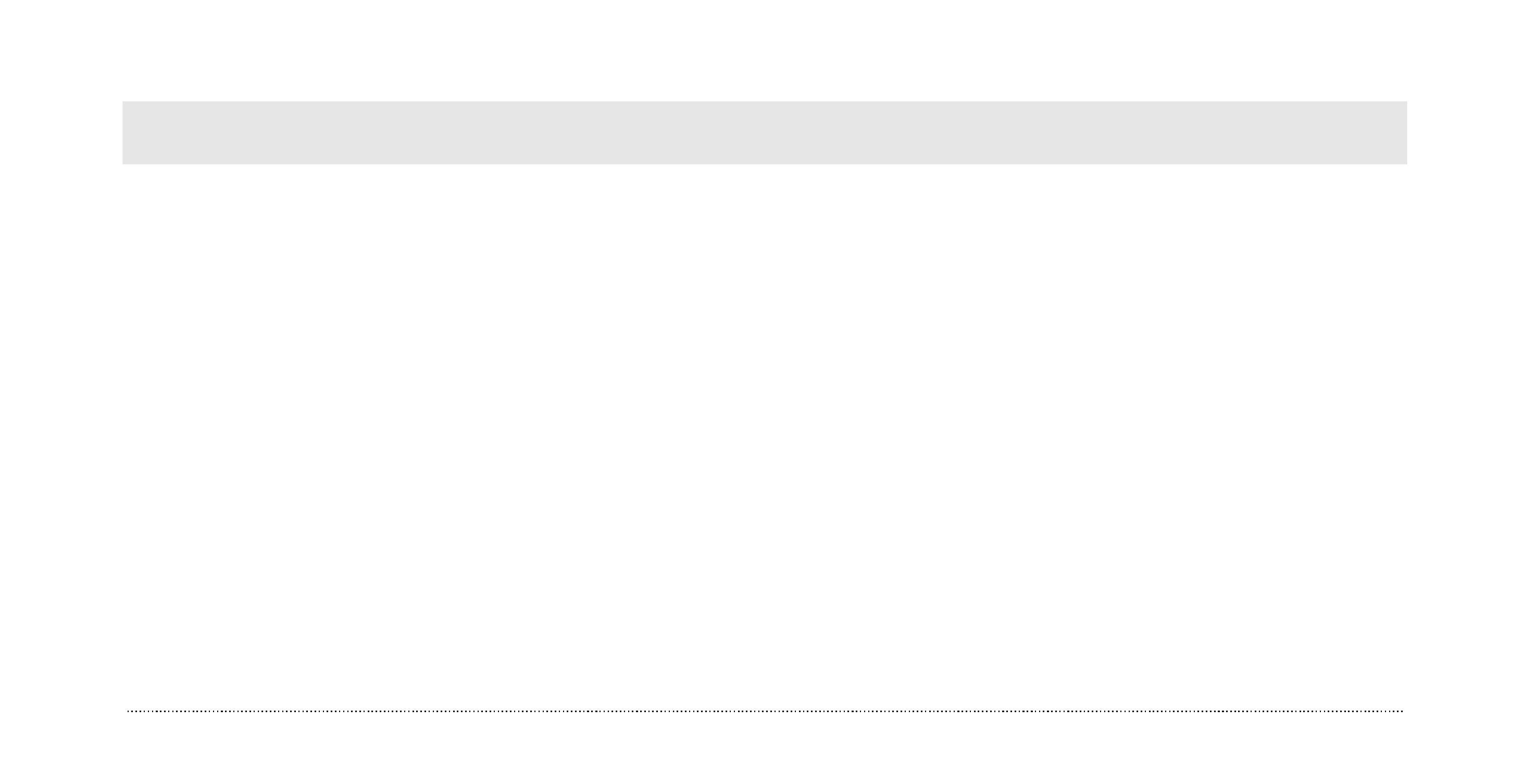}}%
	\put(0.50826347,0.05305313){\color[rgb]{0,0,0}\makebox(0,0)[lt]{\lineheight{1.25}\smash{\begin{tabular}[t]{l}$\fontsize{9pt}{1em}0$\end{tabular}}}}%
	\put(0,0){\includegraphics[width=\unitlength,page=2]{strips_v2.pdf}}%
	\put(0.51484536,0.42513696){\color[rgb]{0,0,0}\makebox(0,0)[lt]{\lineheight{1.25}\smash{\begin{tabular}[t]{l}$\fontsize{9pt}{1em}i\tau_2$\end{tabular}}}}%
	\put(0.7716044,0.42440635){\color[rgb]{0,0,0}\makebox(0,0)[lt]{\lineheight{1.25}\smash{\begin{tabular}[t]{l}$\fontsize{9pt}{1em}S^{-\varepsilon_2}+i\tau_2$\end{tabular}}}}%
	\put(0,0){\includegraphics[width=\unitlength,page=3]{strips_v2.pdf}}%
	\put(0.51559628,0.28556996){\color[rgb]{0,0,0}\makebox(0,0)[lt]{\lineheight{1.25}\smash{\begin{tabular}[t]{l}$\fontsize{9pt}{1em}i\tau_1$\end{tabular}}}}%
	\put(0.78126285,0.28305763){\color[rgb]{0,0,0}\makebox(0,0)[lt]{\lineheight{1.25}\smash{\begin{tabular}[t]{l}$\fontsize{9pt}{1em}S^{-\varepsilon_1}+i\tau_1$\end{tabular}}}}%
	\put(0.80652667,0.35350148){\color[rgb]{0,0,0}\makebox(0,0)[lt]{\lineheight{1.25}\smash{\begin{tabular}[t]{l}$\fontsize{9pt}{1em}W_1$\end{tabular}}}}%
	\put(0.80635942,0.21263556){\color[rgb]{0,0,0}\makebox(0,0)[lt]{\lineheight{1.25}\smash{\begin{tabular}[t]{l}$\fontsize{9pt}{1em}W_0$\end{tabular}}}}%
	\put(0,0){\includegraphics[width=\unitlength,page=4]{strips_v2.pdf}}%
	\put(0.78296729,0.14378447){\color[rgb]{0,0,0}\makebox(0,0)[lt]{\lineheight{1.25}\smash{\begin{tabular}[t]{l}$\fontsize{9pt}{1em}S^{-\varepsilon_0}+i\tau_0$\end{tabular}}}}%
	\put(0,0){\includegraphics[width=\unitlength,page=5]{strips_v2.pdf}}%
	\put(0.80627455,0.4789637){\color[rgb]{0,0,0}\makebox(0,0)[lt]{\lineheight{1.25}\smash{\begin{tabular}[t]{l}$\fontsize{9pt}{1em}W_2$\end{tabular}}}}%
	\put(0.97812036,0.30506502){\color[rgb]{0,0,0}\makebox(0,0)[lt]{\lineheight{1.25}\smash{\begin{tabular}[t]{l}$\fontsize{9pt}{1em}f$\end{tabular}}}}%
	\put(0.08729972,0.1444097){\color[rgb]{0,0,0}\makebox(0,0)[lt]{\lineheight{1.25}\smash{\begin{tabular}[t]{l}$\fontsize{9pt}{1em}h_0(z-i\tau_0)+i\tau_1$\end{tabular}}}}%
	\put(0.08246936,0.35123719){\color[rgb]{0,0,0}\makebox(0,0)[lt]{\lineheight{1.25}\smash{\begin{tabular}[t]{l}$\fontsize{9pt}{1em}h_1(z-i\tau_1)+i\tau_2$\end{tabular}}}}%
	\put(0,0){\includegraphics[width=\unitlength,page=6]{strips_v2.pdf}}%
	\put(0.51373727,0.14629686){\color[rgb]{0,0,0}\makebox(0,0)[lt]{\lineheight{1.25}\smash{\begin{tabular}[t]{l}$\fontsize{9pt}{1em}i\tau_0$\end{tabular}}}}%
	\put(0,0){\includegraphics[width=\unitlength,page=7]{strips_v2.pdf}}%
\end{picture}%
\endgroup%
	\caption{Schematic of the domains and functions in the statement of Theorem \ref{thm:pasting_strips} and Corollary \ref{cor_wd_pasting}. The entire function $f$ will map the domains $S^{-\varepsilon_0}+i\tau_0$ and $W_n$ (in light gray) to a circle around $i \tau_0$, which belongs to an attracting basin, $\mathcal{A}$. The wandering domains $\{U_n\}_{n\geq 1}$ lie in the complement of these domains, and for each $n$, $U_n$ contains $S^{-\eps_n}+i\tau_n$, in dark gray. In white, between $S^{-\eps_n}+i\tau_n$ and $W_n$, the strips where there is no control over $f$ other than an upper bound.}
	\label{fig:idea_pasting}
\end{figure}

The following corollary shows that the function $f$ constructed in Theorem \ref{thm:pasting_strips} has an orbit of unbounded wandering domains, which converge to horizontal strips. Moreover, the translations, $\tau_k$, can be chosen to be large enough, in a precise way, so that the wandering domains are fast escaping.
\begin{cor}\label{cor_wd_pasting} Let $h_k:S\rightarrow S$ be holomorphic maps satisfying that for every $k$, $h_k(S^{-\eps_k})\subset S^{-2\eps_{k+1}}$ and let $f$ be the entire function resulting from Theorem \ref{thm:pasting_strips}. Then, $f$ has an orbit of unbounded escaping wandering domains $\bset{U_k}_{k\geq 1}$ such that for every $k\ge 1$,
\begin{enumerate}
\item \label{item_inclusions}
$S^{-\eps_k}+i\tau_k \subseteq U_k \subseteq S^{+\eps_k}+i\tau_k.$
\item \label{item_fast}
If, in addition, $\tau_k\to \infty$ \emph{fast enough}, namely, if for all $k\geq 1$
\begin{equation}\label{eq_tauk_fast}
	\tau_{k+1}> 4\exp\bb{a_{k-1}\exp\bb{\frac\pi{\eps_{k-1}}\left(\tau_k-\frac{3}{2}\right)}}+\frac32,
\end{equation} then $U_k \subset A(f)$.
\end{enumerate}
\end{cor}
\begin{proof}[Proof of Corollary \ref{cor_wd_pasting}, using Theorem \ref{thm:pasting_strips}] Note that following property \eqref{item:pasting_f_hk} of $f$, for every $k\ge0$ we have 
	$$
	f\bb{S^{-\eps_k}+i\tau_k}\subset\sbb{h_k\bb{S^{-\eps_k}}+i\tau_{k+1}}^{+\delta_k}.
	$$
	For $k=0$, 
	$$
	f(S^{-\eps_0}+i\tau_0)\subset \sbb{h_0\bb{S^{-\eps_0}}+i\tau_{1}}^{+\delta_0}=B\bb{i\tau_0, \delta_0}\subset S^{-\eps_0}+i\tau_0, 
	$$
	implying, by Montel's theorem, that $S^{-\eps_0}+i\tau_0$ is contained in an unbounded Fatou component, $\mathcal{A}$, which must be an attracting basin. For every $k\ge 1$, we get
	$$
	f\bb{S^{-\eps_k}+i\tau_k}\subset\sbb{h_k\bb{S^{-\eps_k}}+i\tau_{k+1}}^{+\delta_k}\subset S^{-\eps_{k+1}}+i\tau_{k+1},
	$$
	since, by assumption, $h_k(S^{-\eps_k})\subset S^{-2\eps_{k+1}}$, and $\delta_k< \eps_{k+1}$. 
	We conclude that the orbit of every point in $S^{-\eps_k}+i\tau_k$ does not intersect the set $\bset{\Ima(z)<0}$. By Montel's theorem, $S^{-\eps_k}+i\tau_k$ is contained in a Fatou component, $U_k$.
	
	A similar argument shows that, following property \eqref{item:pasting_f_i} of the function $f$, for every $z\in W_k$, $f(z)\in B\bb{i\tau_0,\delta_k}\subset \mathcal{A}$. Denote by $A_k$, $k\geq 1$, the Fatou component containing  $W_k$. Observe that since points in $U_k$ escape to infinity and $A_k$ lies between $U_k$ and $U_{k+1}$, all Fatou components $\{A_k\}, \{U_k\}$ are pairwise disjoint. We conclude that 
	$$
	S^{-\eps_k}+i\tau_k\subseteq U_k\subseteq S^{+\eps_k}+i\tau_k.
	$$
	Lastly, since for all $k\ge 1$, $S^{-\eps_k}+i\tau_k\subseteq U_k$ and 
	$$
	f\bb{S^{-\eps_k}+i\tau_k}\subseteq S^{-\eps_{k+1}}+i\tau_{k+1}\subseteq U_{k+1},
	$$
	we see that $f(U_k)\cap U_{k+1}\neq\emptyset.$ Thus, $\{U_k\}_{k\geq 1}$ is an orbit of unbounded escaping wandering domains, proving part  \eqref{item_inclusions} of the corollary.
	
	To see part \eqref{item_fast}, suppose that $f$ results from a sequence $\bset{\tau_k}\subset \R^+$ such that $\tau_0\ge\frac32$ and for every $k\geq 1$, equation \eqref{eq_tauk_fast} holds. Let $R_1\defeq \tau_2-\frac32.$  Note that $\underset{z\in J(f)}\min\abs z<\tau_1-1/2+\eps_1< R_1$ and so for every $r\ge R_1$ we have $M_f(r)> r$, (since otherwise, by Montel's theorem, we would have $B(0,r)\subset F(f)$). Next, for every $k\geq 1$, define $R_{k+1}\defeq M_f^{k}(R_1)$. We will show by induction that
	\begin{equation*}
		R_{k}\le \tau_{k+1}-\frac32.
	\end{equation*}
	For $k=1$ it holds by the way it was defined. Assuming it holds for $R_{k-1}$, and using property \eqref{item:pasting_upperbound} of $f$ and \eqref{eq_tauk_fast},
	\begin{align*}
	R_{k}=M_f(R_{k-1})&\le 4\exp\bb{a_{k-1}\exp\bb{\frac\pi{\eps_{k-1}} R_{k-1}}} \\ &\le 4\exp\bb{a_{k-1}\exp\bb{\frac\pi{\eps_{k-1}} \left(\tau_{k}-\frac{3}{2}\right)}}\\
	&\le \tau_{k+1}-\frac32.
	\end{align*}

	Let $z\in S^{-\eps_1}+i\tau_1 \subset U_1$. Then, by part (1) of the corollary, for every $k\geq 1$, $f^k(z)\in U_{k+1} \subset S^{+\eps_{k+1}}+i\tau_{k+1},$ implying that 
	$$
	\abs{f^k(z)}\ge \abs{\Ima\bb{f^k(z)}}\ge \tau_{k+1}-\frac12-\eps_{k+1}>R_{k}, 
	$$	
	i.e, $z\in A(f)$, see e.g. \cite[Corollary 4.2]{PG_fastescaping_12}. We see that $U_k\subset A(f)$ is a fast escaping wandering domain for all $k\ge 1$.
\end{proof}

In order to prove Theorem \ref{thm_cases_fast}, we include the classification from  \cite{BEFRS_internal22} of simply connected wandering domains, in terms of convergence to the boundary. 
\begin{thm}[{\cite[Theorem C]{BEFRS_internal22}}]\label{thm_internal}
Let $U$ be a simply connected wandering domain of a transcendental entire function~$f$ and let $U_n$ be the Fatou component containing $f^n(U)$, for $n \ge 0$. Then exactly one of the following holds:
	\begin{enumerate}[label=Case (\alph*),leftmargin=\widthof{[Case~(c)]}+\labelsep]
		\item \label{item_internal_a}
		$\liminf_{n\to\infty} \operatorname{dist}(f^{n}(z),\partial U_{n})>0$  for all $z\in U$,
		that is, all orbits \emph{stay away} from the boundary;
		\item \label{item_internal_b} there exists a subsequence $n_k\to \infty$ for which $\operatorname{dist}(f^{n_k}(z),\partial U_{n_k})\to 0$ for all $z\in U$, while for a different subsequence $m_k\to\infty$ we have that
		\[\liminf_{k \to \infty} \operatorname{dist}(f^{m_k}(z),\partial U_{m_k})>0, \quad\text{for }z\in U;\]
		\item \label{item_internal_c}$ \operatorname{dist}(f^{n}(z),\partial U_{n})\to 0$ for all $z\in U$, that is, all orbits \emph{converge to the boundary}.
	\end{enumerate}
\end{thm}

\begin{proof}[Proof of Theorem \ref{thm_cases_fast}]
Let us fix for the rest of the proof the sequence $\eps_k\defeq 8^{-k}$, $k\geq 1$, noting that $\sum_{k=1}^{\infty} \eps_k = 1/7.$ Moreover, given sequences of maps $\{h_k\}$ and translations $\{\tau_k\}$ satisfying the requirements of Theorem \ref{thm:pasting_strips}, for each $n\geq 1$, let $H_n\colon S+i\tau_1\to S+i\tau_{n+1}$ given by
\begin{equation*}
H_n(z)\defeq h_n \circ h_{n-1} \circ\cdots \circ h_1(z-i \tau_1)+i \tau_{n+1}.
\end{equation*}

Next we construct examples of wandering domains for the three cases in Theorem~\ref{thm_internal}. We first prove (a), then (c) and finally (b), which is straightforward combining the arguments used in (a) and (c).
\begin{itemize}[leftmargin=\widthof{[Case~(c)]}+\labelsep]
\item [Case (a)]Suppose first that $h_k(z)\defeq z$ for every $k\geq 1$. Since $2\eps_{k+1}<\eps_{k},$ $h_k(S^{-\eps_k})=S^{-\eps_k}\subset S^{-2\eps_{k+1}}$, and $H_k(z)=z-i\tau_1+i\tau_{k+1}$. Let $\{\tau_k\}$ be a sequence satisfying the requirements of Theorem \ref{thm:pasting_strips}, as well as \eqref{eq_tauk_fast} in Corollary \ref{cor_wd_pasting}. Let $f$ be the entire function resulting from Theorem~\ref{thm:pasting_strips}, with this choice of parameters. Moreover, following Corollary \ref{cor_wd_pasting}, let $\{U_n\}_{n\geq 1}$ be the corresponding orbit of unbounded fast escaping wandering domains. 

We want to show that \ref{item_internal_a}  in Theorem~\ref{thm_internal} holds for all $z\in U_1$. Note that by the trichotomy in the statement of  Theorem~\ref{thm_internal}, it suffices to show that \ref{item_internal_a} holds for one specific point in $U_1$, as then it must hold for all points in $U_1$. Let us choose $z\defeq i \tau_1\in U_1$

Since the maps $h_k$ do not increase Euclidean distances, that is, $\vert h_k(w_1)-h_k(w_2)\vert \leq \vert w_1-w_2\vert$ for all $k\geq 1$, $w_1, w_2\in S$, an inductive argument shows that for all $n\geq 1$,
\begin{align*}
	\vert f^n(z)-H_n(z)\vert & \leq \vert f(f^{n-1}(z))-(h_{n}(f^{n-1}(z)-i\tau_{n})+i\tau_{n+1})\vert+\\
	& \quad + \vert h_{n}(f^{n-1}(z)-i\tau_{n})+i\tau_{n+1} -( h_n(H_{n-1}(z)-i\tau_{n})+i\tau_{n+1})\vert \\
	& \leq  \delta_n+ \vert f^{n-1}(z)-H_{n-1}(z)\vert \leq  \sum^n_{k=1} \delta_k \leq \sum^n_{k=1} \eps_{k+1}\leq\frac{1}{56},
\end{align*}
where we have used \eqref{eq_alphak} and property \eqref{item:f_small_h} of $f$. Since, by Corollary \ref{cor_wd_pasting}, $S^{-\eps_n}+i\tau_n \in U_n$,  we have
\begin{equation*}
	\dist(f^n(z),\partial U_{n+1}) \geq \dist(H_n(z),\partial U_{n+1})-\frac{1}{56} \geq \frac{1}{2}-\eps_n-\frac{1}{56}>\frac{1}{5}>0,
\end{equation*} 
as desired.
\item [Case (c)] Next, for every $k\geq 1$, define $h_k(z)\defeq \frac{i}{2}-3\eps_{k+1}$, so that $h_k(S^{-\eps_k})=\frac{i}{2}-3\eps_{k+1}\subset S^{-2\eps_{k+1}}$ and $H_k(z)=h_k(z)+i\tau_{k+1}$ in $S^{-\eps_1}$. We can again choose a sequence $\{\tau_k\}$ satisfying the requirements of Theorem~\ref{thm:pasting_strips}, as well as \eqref{eq_tauk_fast} in Corollary \ref{cor_wd_pasting},  to obtain an entire function $f$ with an orbit $\{U_n\}_{n\geq 1}$ of fast escaping unbounded wandering domains.

Observe that this time, since $h_k(z)$ is a constant function for every $k$, instead of to a sum of errors, property  \eqref{item:f_small_h} of $f$ leads to 
$$\vert f^n(z)-H_n(z)\vert\leq \delta_n\leq \eps_n \quad \text{ for every } n\geq 1.$$

To see that \ref{item_internal_c}  in Theorem~\ref{thm_internal} holds, choose $z\defeq i \tau_1\in U_1$. Since $i/2\in \partial S$, and by Corollary \ref{cor_wd_pasting}, $U_n \subseteq S^{+\eps_n}+i\tau_n$, we have that 
\begin{equation*}
	\dist(f^n(z),\partial U_{n+1})\leq \dist(f^n(z), H_n(z))+  \dist(H_n(z),\partial U_{n+1})\leq \eps_n+4\eps_n=5\eps_n,
\end{equation*} 
which tends to $0$ as $n\to \infty$, as we wanted to show.
\item [Case (b)] Finally, a wandering domain satisfying \ref{item_internal_b} can be constructed using the sequence of maps
\begin{equation*}
	h_k(z)\defeq \begin{cases}
		\frac{i}{2}-3\eps_{k},& k=2m+1 \text{ for some }m\in \N,\\
		0,& \text{otherwise}.
	\end{cases}
\end{equation*}
Indeed, let $f$ and $\{U_n\}_{n\geq 1}$ be constructed as before, and choose $z\defeq i \tau_1 \in U_1$. Then, arguing as in our first example, one can show that $\dist(f^{2n}(z), \partial U_{2n+1})$ is uniformly bounded away from zero, while $\dist(f^{2n+1}(z), \partial U_{2n+2})\to 0$ as $n\to \infty$, arguing this time as in our second example. \qedhere
\end{itemize}
\end{proof}

\begin{observation}\label{obs:contracting}In \cite[Theorem A]{BEFRS_internal22}, wandering domains are classified into \textit{contracting}, \textit{semi-contracting} and \textit{eventually isometric} in terms of hyperbolic distances. Our method can be used to construct unbounded fast escaping wandering domains of contracting type with any of the three behaviours described in Theorem \ref{thm_internal}. A simply connected wandering domain $U_1$ is contracting if $\dist_{U_{n+1}}(f^n(z), f^n(w))\to 0$ for all $z,w,\in U_1$. In our setting, it suffices to show that if $T_{n+1}\defeq S^{-\eps_{n+1}}+i\tau_{n+1}$, then  $\dist_{T_{n+1}}(f^n(z), f^n(w))\to 0$ for some $z,w\in U_1$. In case (c) in the previous proof, if $\tau_n$ is chosen large enough such that $\frac{\delta_n}{2\eps_{n+1}}\to 0$, then since $h_n(z)=\frac{i}{2}-3\eps_{n+1}$ we have
\begin{eqnarray*}
\dist_{T_{n+1}}(f^n(z), f^n(w)) &\leq& \dist_{T_{n+1}}(f^n(z),H_n(z))+\dist_{T_{n+1}}(H_n(z), H_n(w))\\
& & \quad +\dist_{T_{n+1}}(H_n(w), f^n(w))\nonumber \\
&\leq& 2\log\left(1+ \frac{|f^n(z)-H_n(z)|}{\min \{\dist(f^n(z),\partial T_{n+1}),\dist(H^n(z),\partial T_{n+1})\}}\right) \nonumber \\ & & \quad + 2\log\left(1+ \frac{|f^n(w)-H_n(w)|}{\min \{\dist(f^n(w),\partial T_{n+1}),\dist(H_n(w),\partial T_{n+1})\}}\right)\nonumber \\
&\leq&  4\log \left(1+ \frac{\delta_n}{2\eps_{n+1}}\right) \to 0, \nonumber
\end{eqnarray*}
for $n \to \infty$, where the second inequality is deduced by \cite[p.157]{shapiro}. 
Similarly, case~(b) in the previous proof gives a contracting wandering domain whenever parameters are chosen so that $\frac{\delta_k}{2\eps_{k+1}}\to 0$. To obtain a contracting wandering domain where orbits stay away from the boundary, we can use the sequence $h_k(z)=0$ for all $k\geq 1$ (instead of the sequence $h_k(z)=z$ that was used in case (a) above).

Our method only provides an approximation of our function  on an unbounded subset strictly contained in the wandering domain. This prevents us from constructing examples of eventually isometric type. Also, for the semi-contracting case one would need an approach similar to the above, and in particular, the sum of errors (see top of p.8) to shrink. In our case, this can only be achieved using contracting model maps, which would then result to a contracting wandering domain.
\end{observation}

\begin{proof}[Proof of Theorem \ref{thm_fast_escaping}] It is a direct consequence of Theorem \ref{thm_cases_fast}.
\end{proof}

The rest of this section is devoted to proving Theorem \ref{thm:pasting_strips}. To be able to define the entire function $f$, that approximates the different maps, $h_k$, in disjoint regions, we will use H\"ormander's theorem on the solution to a $\bar\partial$-equation. 
In what follows, subharmonic functions play an important role. For a definition and basic properties see \cite[Chapter~2]{hayman_subharmonic}.
\begin{Hormander}[{\cite[Theorem~4.2.1]{Hormander}}]\label{thm:Hormander}
	Let $u:\C\rightarrow\R$ be a subharmonic function. Then, for every locally integrable function $g$ there is a solution $\alpha$ of the equation $\bar\partial \alpha=g$ such that
	\begin{equation}\label{eq:Hormander}
		\int_\C\abs {\alpha(z)}^2\frac{e^{-u(z)}}{\bb{1+\abs z^2}^{2}}dm(z)\le\frac12\int_\C\abs {g(z)}^2e^{-u(z)}dm(z),
	\end{equation}
	provided that the integral on the right hand side is finite.
\end{Hormander}

Roughly speaking, the proof of Theorem \ref{thm:pasting_strips} will proceed as follows: the entire function $f$ will aim to approximate a \textit{model map} $h$ that is equal to $h_k$ on most of a translated copy of S, $S+i\tau_k,$ and to $i\tau_0$ on $W_k$. Our entire map will then be $f=h-\alpha$, where $\alpha$ results from applying Hörmander's theorem to the equation $g= \overline\partial h$ and to a subharmonic function $u$, carefully chosen. In particular, $f$ will be entire, as
\begin{equation}\label{eq_f_entire}
\overline\partial f(z)=\overline\partial\bb{ h(z)-\alpha(z)}=\overline\partial h(z)-\overline\partial\alpha(z)=\overline\partial h(z)-g(z)=0.
\end{equation}
More precisely, the function $u$ interacts with $f$ in the following way.
On one hand, we would like to choose $u$ as negative as possible where we want a good approximation of the model map by $f$ (Lemma \ref{lem:sh_stripes} \ref{item_u_pasting_2}\ref{item_u_pasting_22}). Moreover, anywhere in the plane, $|f|$ will be bounded roughly by $\exp \bb{u}$ and so one would want to make $u$ as small as possible to have a minimal possible bound on $|f|$ (Lemma \ref{lem:sh_stripes}\ref{item_u_pasting_1}). On the other hand, in order for the integral on the right hand side in H\"ormander's theorem to converge, $u$ will need to be large in the support of $g$ (Lemma \ref{lem:sh_stripes} \ref{item_u_pasting_2}\ref{item_u_pasting_21}). We start by constructing the subharmonic function $u$:

\begin{lemma}\label{lem:sh_stripes}
Given a sequence $\tau_k \nearrow\infty$ with $\tau_0\ge \frac32$ and $\tau_{k+1}\ge\tau_{k}+1+\eps_k+\eps_{k+1}$, a sequence of weights $a_k\nearrow \infty$ with $a_0\ge10\cdot \tau_0$, and a sequence $\eps_k \searrow 0$, there exists a subharmonic function $u\colon \C\to \R$ so that
	\begin{enumerate}[label=(\arabic*)]
		\item \label{item_u_pasting_1}If $\Ima(z)\le\tau_{k+1}-\frac12$, then
		$$
		u(z)\le 2a_k\exp\bb{\frac{\pi}{\eps_k}\abs{\Rea(z)}};
		$$
		\item \label{item_u_pasting_2} For every $z$ with $\Ima(z)>\tau_0-\frac12+\frac{\eps_0}{4},$
		\begin{enumerate}[label=(\roman*)]
			\item \label{item_u_pasting_21} If $z\in \bb{S^{-\frac{\eps_k}4}\setminus S^{-\frac{3\eps_k}4}}+i \tau_k$ for some $k\geq 0$, then
			$$
			u(z)\ge \frac{a_k}3\exp\bb{\frac{\pi}{\eps_k}\vert\Rea(z)\vert}-8\log\abs z;
			$$
			\item \label{item_u_pasting_22} If $z\nin \bunion k 0 \infty\bb{S\setminus S^{-\eps_k}}+i \tau_k$, then
			$$
			u(z)= -8\log\abs z.
			$$
		\end{enumerate}
	\end{enumerate}
\end{lemma}

\begin{proof}
For every $\eps\in\bb{0,1}$, define $v_\eps\colon \C \to \R$ by
	$$
	v_\eps(z)=v_\eps(x+iy)\defeq
		\begin{cases}
		\cosh\bb{\frac\pi\eps x}\cdot\cos\bb{\frac\pi\eps y}, & \abs y<\frac\eps2,\\
		0,& \text{otherwise}.
	\end{cases}
	$$
	This function is subharmonic as a local maximum of subharmonic functions. (Note that a continuous function $u$ is subharmonic if and only if it satisfies the mean value inequality and the later is preserved under taking maximum). If $\abs y<\frac\eps 4$, then, by monotonicity of the cosine function,
	\begin{equation*}\label{eq_v_eps}
	v_\eps(z)\ge\cosh\bb{\frac\pi\eps x}\cdot\cos\bb{\frac\pi\eps \cdot\frac\eps 4}\ge  \frac1{\sqrt2}\cosh\bb{\frac\pi\eps x}\ge \frac{e^{\frac{\pi}\eps x}}3,
	\end{equation*}
	where we have used that 
	\begin{equation}\label{eq_cosh}
	\frac{\exp(x)}{2}<\cosh(x)\leq 2 \exp(x) \quad \text{ for all } x>0.
	\end{equation}	
	
	For each $k\geq 0$, let $\omega_k\defeq\frac{1-\eps_k}2$ and define $u_k\colon \C \to \R$ as
	$$
	u_k(z)\defeq v_{\eps_k}\bb{z-i\cdot\bb{\tau_k-\omega_k}}+v_{\eps_k}\bb{z-i\cdot\bb{\tau_k+\omega_k}},
	$$
	noting that it is supported in $(S\setminus S^{-\eps_k})+i\tau_k$.
	
	Since the map $z\mapsto-8\log\abs z$ is not well defined at the origin, we will `attach it' to $u_0$ by using the fact that the local maximum of subharmonic functions is a subharmonic function. We define the function
	$$
	\tilde{u}_0(z)\defeq \begin{cases}
		0,& \Ima(z)\le\tau_0-\frac12,\\
		\max\bset{0,a_0\cdot u_0(z)-8\log\abs z},&  \Ima(z)-\tau_0+\frac12\in\bb{0,\frac{\eps_0}4},\\
		a_0\cdot u_0(z)-8\log\abs z,& \text{otherwise}.
	\end{cases}
	$$
	In each of the domains the function $\tilde{u}_0$ is subharmonic. It is left to see it is well defined on their boundaries. Note that if $\Ima(z)= \tau_0-\frac12$, then
	$$
	\abs{\Ima(z-i\bb{\tau_0-\omega_0})}=\tau_0-\omega_0-\Ima(z)= \tau_0-\omega_0-\bb{\tau_0-\frac12}=\frac{\eps_0}2,
	$$
	implying that $v_{\eps_0}\bb{z-i\bb{\tau_0-\omega_0}}=0$. On the other hand, $\abs z\ge\abs{\Ima(z)}=\tau_0-\frac12\ge 1$ and therefore
$$
		a_0\cdot  v_{\eps_0}\bb{z-i\bb{\tau_0-\omega_0}}-8\log\abs z=-8\log\abs z\le0.
$$
	If $\Ima(z)=\tau_0-\frac12+\frac{\eps_0}4$, then since $a_0>10\cdot\tau_0$,
	\begin{align*}
		a_0\cdot v_{\eps_0}\bb{z-i\bb{\tau_0-\omega_0}}-8\log\abs z&\ge a_0\cdot \cosh\bb{\frac\pi{\eps_0} x}\cos\bb{\frac\pi4}-4\log\bb{x^2+\tau_0^2}\\
		=& \frac{a_0}{\sqrt 2}\cdot \cosh\bb{\frac\pi{\eps_0} x}-4\log\bb{x^2+\tau_0^2}>0.
	\end{align*}
	Thus, $\tilde{u}_0$ is well defined and subharmonic as a local maximum of subharmonic functions.
	
	Define the function $u\colon \C\to \R$ as
	$$
	u(z)\defeq \tilde{u}_0(z)+\sum_{k=1}^\infty a_k\cdot u_k(z), 
	$$
To see Property \ref{item_u_pasting_1} holds, note that if  $\Ima(z)\le\tau_{k+1}-\frac12$, then for every $\nu\ge k+1$, $z$ does not intersect the support of $u_\nu$. On the other hand, since the strips are disjoint, there exists at most one index, $0\leq \nu_0\le k$, such that $z$ is in the support of $u_{\nu_0}$ and since the sequence $\bset{a_k}$ is monotone increasing and $\bset{\eps_k}$ is monotone decreasing, using \eqref{eq_cosh},
$$
u(z)\leq a_{\nu_0}\cdot u_{\nu_0}(z)\le a_{\nu_0}\cosh\bb{\frac\pi{\eps_{\nu_0}}\abs{\Rea(z)}}\le a_k\cosh\bb{\frac\pi{\eps_k}\abs{\Rea(z)}}\le 2a_k\exp\bb{\frac\pi{\eps_k}\abs{\Rea(z)}}.
$$
To see Property \ref{item_u_pasting_2}, note that if $z$ belongs to the support of one of the functions $u_k$, then $k$ is unique and either $\Ima(z)\in\bb{\tau_k-\frac12+\frac{\eps_k}4,\tau_k-\frac12+\frac{3\eps_k}4}$ or $\Ima(z)\in\bb{\tau_k+\frac12-\frac{3\eps_k}4,\tau_k+\frac12-\frac{\eps_k}4}$. In the first case we see that there exists $y\in\bb{\frac{\eps_k}4,\frac{3\eps_k}4}$ satisfying
\begin{align*}
u(z)&=a_k\cdot u_k(z)-8\log\abs z=a_k\cdot \cos\bb{\frac\pi{\eps_k}\cdot \bb{\omega_k-\frac12+y}}\cosh\bb{\frac{\pi}{\eps_k}\cdot\Rea(z)}-8\log\abs z\\
&\ge a_k\cdot \cos\bb{\frac\pi{4}}\cosh\bb{\frac{\pi}{\eps_k}\cdot\Rea(z)}-8\log\abs z\ge\frac{\sqrt{2} a_k}{2}\exp\bb{\frac{\pi}{\eps_k}\cdot\abs{\Rea(z)}}-8\log\abs z.
\end{align*}
The second case is identical with $(-y)$ replacing $y$. Lastly, if $z$ does not belong to any strip $\bb{S\setminus S^{-\eps_k}}+i\tau_k$, then it does not belong to the support of any function $u_k$, $k\geq 0$, and therefore
$$
u(z)=\tilde{u}_0(z)=-8\log\abs z,
$$
concluding the proof of the lemma.\end{proof}

 \begin{figure}[htp]
  \begin{center}
  \def\svgwidth{\linewidth}
 \begingroup%
 \makeatletter%
 \providecommand\color[2][]{%
 	\errmessage{(Inkscape) Color is used for the text in Inkscape, but the package 'color.sty' is not loaded}%
 	\renewcommand\color[2][]{}%
 }%
 \providecommand\transparent[1]{%
 	\errmessage{(Inkscape) Transparency is used (non-zero) for the text in Inkscape, but the package 'transparent.sty' is not loaded}%
 	\renewcommand\transparent[1]{}%
 }%
 \providecommand\rotatebox[2]{#2}%
 \newcommand*\fsize{\dimexpr\f@size pt\relax}%
 \newcommand*\lineheight[1]{\fontsize{\fsize}{#1\fsize}\selectfont}%
 \ifx\svgwidth\undefined%
 \setlength{\unitlength}{711.49606299bp}%
 \ifx\svgscale\undefined%
 \relax%
 \else%
 \setlength{\unitlength}{\unitlength * \real{\svgscale}}%
 \fi%
 \else%
 \setlength{\unitlength}{\svgwidth}%
 \fi%
 \global\let\svgwidth\undefined%
 \global\let\svgscale\undefined%
 \makeatother%
 \begin{picture}(1,0.43824701)%
 	\lineheight{1}%
 	\setlength\tabcolsep{0pt}%
 	\put(0,0){\includegraphics[width=\unitlength,page=1]{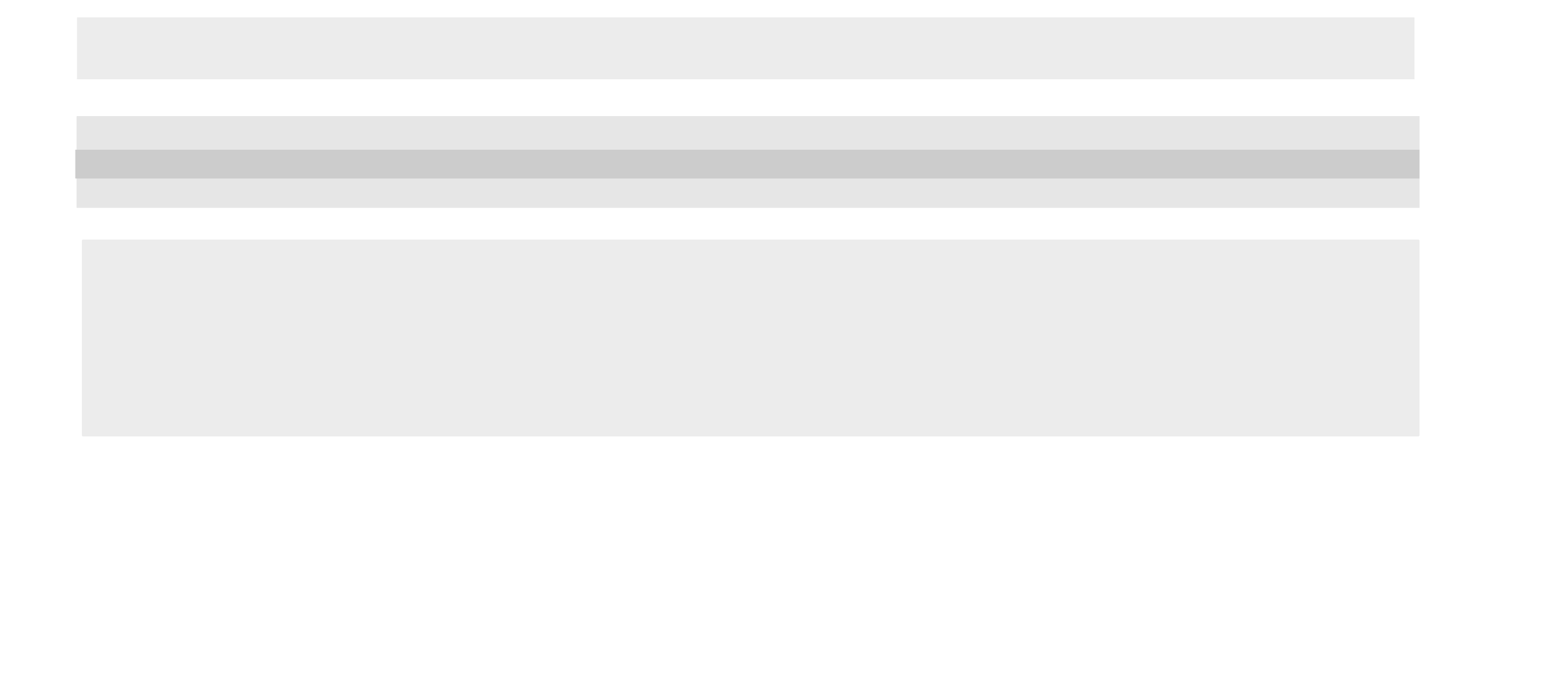}}%
 	\put(-0.02845311,0.3509626){\color[rgb]{0,0,0}\makebox(0,0)[lt]{\smash{\begin{tabular}[t]{l}$-\frac{\varepsilon_k}{4}$\end{tabular}}}}%
 	\put(0,0){\includegraphics[width=\unitlength,page=2]{u_chi2.pdf}}%
 	\put(0.99076561,0.33604033){\color[rgb]{0,0,0}\makebox(0,0)[lt]{\lineheight{1.25}\smash{\begin{tabular}[t]{l}$-\frac{3\varepsilon_k}{8}$\end{tabular}}}}%
 	\put(0.98460697,0.09429469){\color[rgb]{0,0,0}\makebox(0,0)[lt]{\lineheight{1.25}\smash{\begin{tabular}[t]{l}$-\frac{3\varepsilon_k}{8}$\end{tabular}}}}%
 	\put(0.92811474,0.31832945){\color[rgb]{0,0,0}\makebox(0,0)[lt]{\lineheight{1.25}\smash{\begin{tabular}[t]{l}$-\frac{5\varepsilon_k}{8}$\end{tabular}}}}%
 	\put(0,0){\includegraphics[width=\unitlength,page=3]{u_chi2.pdf}}%
 	\put(0.70105514,0.39977246){\color[rgb]{0,0,0}\makebox(0,0)[lt]{\lineheight{1.25}\smash{\begin{tabular}[t]{l}$\fontsize{9pt}{1em}\chi\equiv 1$\end{tabular}}}}%
 	\put(0.70708688,0.21166457){\color[rgb]{0,0,0}\makebox(0,0)[lt]{\lineheight{1.25}\smash{\begin{tabular}[t]{l}$\fontsize{9pt}{1em}\chi\equiv 1$\end{tabular}}}}%
 	\put(0.70114762,0.3255987){\color[rgb]{0,0,0}\makebox(0,0)[lt]{\lineheight{1.25}\smash{\begin{tabular}[t]{l}$\fontsize{7pt}{1em}\chi\equiv 0$\end{tabular}}}}%
 	\put(0.45622422,0.3286281){\color[rgb]{0,0,0}\makebox(0,0)[lt]{\lineheight{1.25}\smash{\begin{tabular}[t]{l}$\fontsize{9pt}{1em}\nabla\chi\neq 0$\end{tabular}}}}%
 	\put(0.14199344,0.20904093){\color[rgb]{0,0,0}\makebox(0,0)[lt]{\lineheight{1.25}\smash{\begin{tabular}[t]{l}$\fontsize{9pt}{1em}u(z)=-8\log\vert z \vert$\end{tabular}}}}%
 	\put(0.1510405,0.32757834){\color[rgb]{0,0,0}\makebox(0,0)[lt]{\lineheight{1.25}\smash{\begin{tabular}[t]{l}$\fontsize{9pt}{1em}u(z)\!\gg \! 0$\end{tabular}}}}%
 	\put(0.14032636,0.40400803){\color[rgb]{0,0,0}\makebox(0,0)[lt]{\lineheight{1.25}\smash{\begin{tabular}[t]{l}$\fontsize{9pt}{1em}u(z)=-8\log\vert z \vert$\end{tabular}}}}%
 	\put(-0.0301435,0.2957886){\color[rgb]{0,0,0}\makebox(0,0)[lt]{\lineheight{1.25}\smash{\begin{tabular}[t]{l}$-\frac{3\varepsilon_k}{4}$\end{tabular}}}}%
 	\put(0,0){\includegraphics[width=\unitlength,page=4]{u_chi2.pdf}}%
 	\put(-0.02885872,0.07557819){\color[rgb]{0,0,0}\makebox(0,0)[lt]{\smash{\begin{tabular}[t]{l}$-\frac{\varepsilon_k}{4}$\end{tabular}}}}%
 	\put(-0.04823262,0.2807334){\color[rgb]{0,0,0}\makebox(0,0)[lt]{\lineheight{1.25}\smash{\begin{tabular}[t]{l}$-\varepsilon_k$\end{tabular}}}}%
 	\put(0,0){\includegraphics[width=\unitlength,page=5]{u_chi2.pdf}}%
 	\put(0.92770902,0.11490632){\color[rgb]{0,0,0}\makebox(0,0)[lt]{\lineheight{1.25}\smash{\begin{tabular}[t]{l}$-\frac{5\varepsilon_k}{8}$\end{tabular}}}}%
 	\put(0,0){\includegraphics[width=\unitlength,page=6]{u_chi2.pdf}}%
 	\put(0.70064954,0.03405589){\color[rgb]{0,0,0}\makebox(0,0)[lt]{\lineheight{1.25}\smash{\begin{tabular}[t]{l}$\fontsize{9pt}{1em}\chi\equiv 1$\end{tabular}}}}%
 	\put(0.70074194,0.10401357){\color[rgb]{0,0,0}\makebox(0,0)[lt]{\lineheight{1.25}\smash{\begin{tabular}[t]{l}$\fontsize{7pt}{1em}\chi\equiv 0$\end{tabular}}}}%
 	\put(0.4558186,0.10210198){\color[rgb]{0,0,0}\makebox(0,0)[lt]{\lineheight{1.25}\smash{\begin{tabular}[t]{l}$\fontsize{9pt}{1em}\nabla\chi\neq 0$\end{tabular}}}}%
 	\put(0.15063491,0.1042574){\color[rgb]{0,0,0}\makebox(0,0)[lt]{\lineheight{1.25}\smash{\begin{tabular}[t]{l}$\fontsize{9pt}{1em}u(z)\!\gg \! 0$\end{tabular}}}}%
 	\put(0.13992077,0.02782772){\color[rgb]{0,0,0}\makebox(0,0)[lt]{\lineheight{1.25}\smash{\begin{tabular}[t]{l}$\fontsize{9pt}{1em}u(z)=-8\log\vert z \vert$\end{tabular}}}}%
 	\put(-0.03054911,0.13093515){\color[rgb]{0,0,0}\makebox(0,0)[lt]{\lineheight{1.25}\smash{\begin{tabular}[t]{l}$-\frac{3\varepsilon_k}{4}$\end{tabular}}}}%
 	\put(0,0){\includegraphics[width=\unitlength,page=7]{u_chi2.pdf}}%
 	\put(-0.0541158,0.15422939){\color[rgb]{0,0,0}\makebox(0,0)[lt]{\lineheight{1.25}\smash{\begin{tabular}[t]{l}$-\varepsilon_k$\end{tabular}}}}%
 	\put(0,0){\includegraphics[width=\unitlength,page=8]{u_chi2.pdf}}%
 \end{picture}%
 \endgroup%
	\caption{Schematic of the values that the subharmonic function $u$ from Lemma \ref{lem:sh_stripes} and the function $\chi$ from the proof of Theorem \ref{thm:pasting_strips} take in the strip $S+i\tau_k$ for $k\geq 1$.}
	\label{fig:sub_sec2}
	\end{center}
\end{figure}
In our calculations, we will use the following observation:
	\begin{obs}\label{obs:int_bnd}
		For every polynomial $P$ and every $\eps>0$ there exists $a_0=a_0(P,\eps)>0$ such that for all $a\geq a_0$, 
		$$
		\integrate 0\infty {\abs{P(x)}\exp(-a\cdot x)}x<\eps.
		$$
	\end{obs}

\begin{proof}[Proof of Theorem \ref{thm:pasting_strips}] We start with the following claim.
\begin{Claim}We can choose a smooth map $\chi\colon \C\rightarrow[0,1]$ so that:
	\begin{enumerate}[label=(\roman*)]
		\item \label{item:chi_11}For every $z\nin \bunion k 0 \infty\bb{S^{-\frac{\eps_k}4}\setminus S^{-\frac{3\eps_k}{4}}}+i \tau_k$, $\chi(z)=1$.
		\item \label{item:chi_12} For every $z\in \bunion k 0 \infty\bb{S^{-\frac{3\eps_k}8}\setminus S^{-\frac{5\eps_k}8}}+i \tau_k$, $\chi(z)=0.$
		\item \label{item:chi_13} $\nabla\chi$ is supported on $\bunion k 0\infty \bb{S^{-\frac{\eps_k}4}\setminus S^{-\frac{3\eps_k}4}}+i \tau_k$ and there exists a numerical constant $C>1$ so that for every $k\ge 0$,
		$$
		\underset{z\in \bb{S^{-\frac{\eps_k}{4}}\setminus S^{-\frac{3\eps_k}{4}}}+i \tau_k}\sup\;\abs{\nabla\chi(z)}\le \frac{C}{\eps_k}.
		$$
	\end{enumerate}
\end{Claim}
\begin{subproof}

Recall the \textit{bump function} $b:\C\rightarrow[0,\infty)$ defined by 
\begin{equation}\label{eq_bump}
b(z)\defeq	\begin{cases}
	A\exp\bb{\frac{-1}{1-\abs z^2}},& \abs z<1,\\
	0,& \abs z\ge 1,
\end{cases}
\end{equation}
where the constant $A$ is chosen so that $\int b(z)dm(z)=1$.  Note that this function is supported on the unit disk. For every $k\geq 0$, let $t_k$ be the convolution of the function $\alpha(z)\defeq\frac{32^2}{\eps_k^2}\cdot b\bb{\frac{32z}{\eps_k}}$, supported on $\bset{\abs z<\frac{\eps_k}{32}}$, with the function $\beta$ defined by
$$\beta(z)=\beta(x+iy)\defeq
\begin{cases}
	1,& z\nin \bunion k 0\infty S^{-\frac{9\eps_k}{32}}\setminus S^{-\frac{23\eps_k}{32}} ,\\
	0,& z\in  \bunion k 0\infty S^{-\frac{11\eps_k}{32}}\setminus S^{-\frac{21\eps_k}{32}}, \\
	-\frac{16}{\eps_k}\bb{y+\frac12}+\frac{11}{2} ,&  y\in\bb{-\frac12+\frac{9\eps_k}{32},-\frac12+\frac{11\eps_k}{32}},\\
	\frac{16}{\eps_k}\bb{y+\frac12}-\frac{21}2,& y\in\bb{-\frac12+\frac{21\eps_k}{32},-\frac12+\frac{23\eps_k}{32}} ,\\
	-\frac{16}{\eps_k}\bb{y-\frac12}-\frac{21}2,&  y\in\bb{\frac12-\frac{23\eps_k}{32},\frac12-\frac{21\eps_k}{32}},\\
	\frac{16}{\eps_k}\bb{y-\frac12}+\frac{11}{2},&  y\in\bb{\frac12-\frac{11\eps_k}{32},\frac12-\frac{9\eps_k}{32}},\\
\end{cases}
$$
see Figure \ref{fig:beta}.
\begin{figure}[htp]
	\centering
	\def\svgwidth{\linewidth}
\begingroup%
\makeatletter%
\providecommand\color[2][]{%
	\errmessage{(Inkscape) Color is used for the text in Inkscape, but the package 'color.sty' is not loaded}%
	\renewcommand\color[2][]{}%
}%
\providecommand\transparent[1]{%
	\errmessage{(Inkscape) Transparency is used (non-zero) for the text in Inkscape, but the package 'transparent.sty' is not loaded}%
	\renewcommand\transparent[1]{}%
}%
\providecommand\rotatebox[2]{#2}%
\newcommand*\fsize{\dimexpr\f@size pt\relax}%
\newcommand*\lineheight[1]{\fontsize{\fsize}{#1\fsize}\selectfont}%
\ifx\svgwidth\undefined%
\setlength{\unitlength}{717.16535433bp}%
\ifx\svgscale\undefined%
\relax%
\else%
\setlength{\unitlength}{\unitlength * \real{\svgscale}}%
\fi%
\else%
\setlength{\unitlength}{\svgwidth}%
\fi%
\global\let\svgwidth\undefined%
\global\let\svgscale\undefined%
\makeatother%
\begin{picture}(1,0.45849802)%
	\lineheight{1}%
	\setlength\tabcolsep{0pt}%
	\put(0,0){\includegraphics[width=\unitlength,page=1]{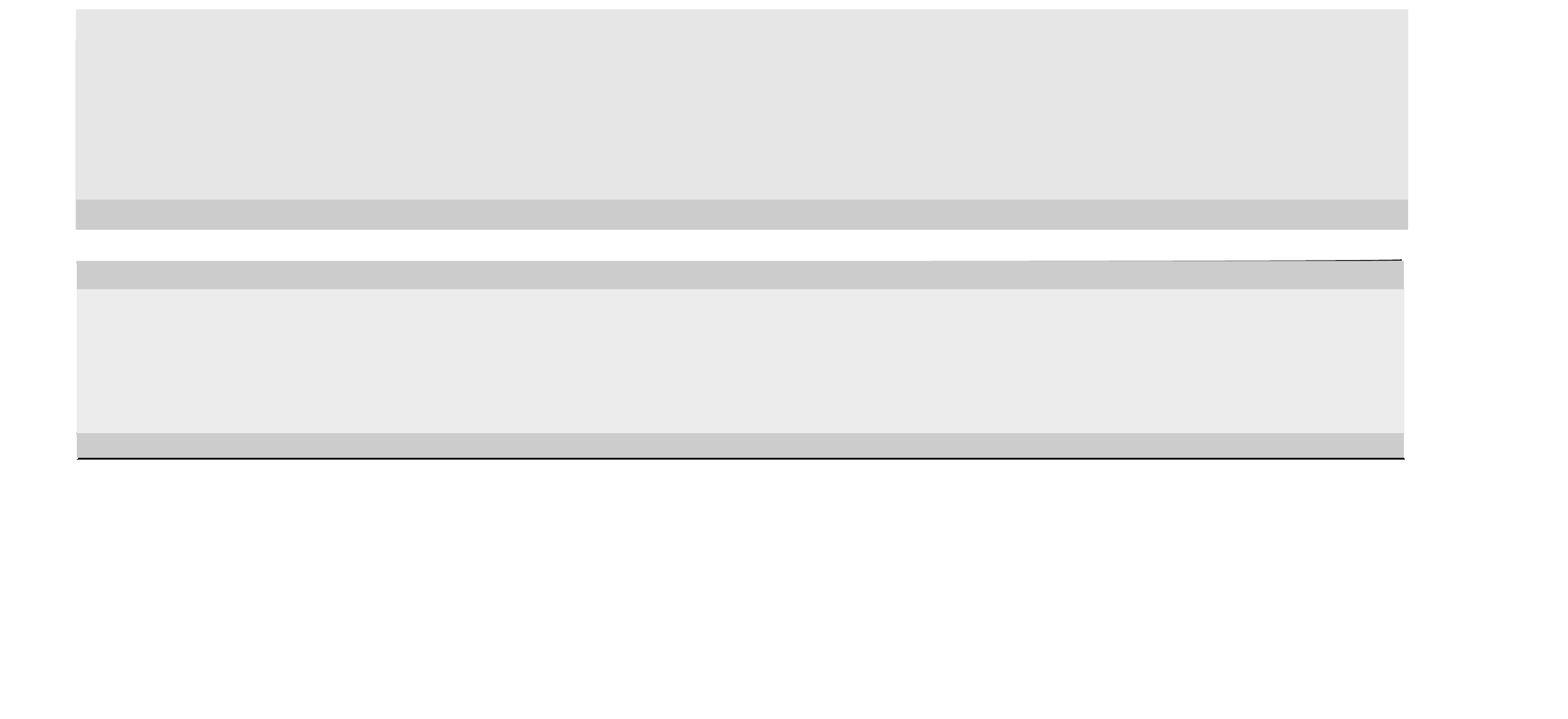}}%
	\put(0.31304802,0.10775884){\color[rgb]{0,0,0}\makebox(0,0)[lt]{\lineheight{1.25}\smash{\begin{tabular}[t]{l}$\color{red}\chi=1$\end{tabular}}}}%
	\put(-0.02823735,0.32675393){\color[rgb]{0,0,0}\makebox(0,0)[lt]{\smash{\begin{tabular}[t]{l}$-\frac{\varepsilon_k}{4}$\end{tabular}}}}%
	\put(-0.03246551,0.2665725){\color[rgb]{0,0,0}\makebox(0,0)[lt]{\lineheight{1.25}\smash{\begin{tabular}[t]{l}$-\frac{3\varepsilon_k}{8}$\end{tabular}}}}%
	\put(-0.03218985,0.17718533){\color[rgb]{0,0,0}\makebox(0,0)[lt]{\lineheight{1.25}\smash{\begin{tabular}[t]{l}$-\frac{5\varepsilon_k}{8}$\end{tabular}}}}%
	\put(0,0){\includegraphics[width=\unitlength,page=2]{chi-beta.pdf}}%
	\put(0.9758154,0.30762372){\color[rgb]{0,0,0}\makebox(0,0)[lt]{\lineheight{1.25}\smash{\begin{tabular}[t]{l}$-\frac{9\varepsilon_k}{32}$\end{tabular}}}}%
	\put(0.92931956,0.28724943){\color[rgb]{0,0,0}\makebox(0,0)[lt]{\lineheight{1.25}\smash{\begin{tabular}[t]{l}$-\frac{11\varepsilon_k}{32}$\end{tabular}}}}%
	\put(0.92094939,0.16253835){\color[rgb]{0,0,0}\makebox(0,0)[lt]{\lineheight{1.25}\smash{\begin{tabular}[t]{l}$-\frac{21\varepsilon_k}{32}$\end{tabular}}}}%
	\put(0.97418065,0.14174727){\color[rgb]{0,0,0}\makebox(0,0)[lt]{\lineheight{1.25}\smash{\begin{tabular}[t]{l}$-\frac{23\varepsilon_k}{32}$\end{tabular}}}}%
	\put(0,0){\includegraphics[width=\unitlength,page=3]{chi-beta.pdf}}%
	\put(0.1457129,0.36979075){\color[rgb]{0,0,0}\makebox(0,0)[lt]{\lineheight{1.25}\smash{\begin{tabular}[t]{l}$\chi\equiv 1$\end{tabular}}}}%
	\put(0.71838727,0.36059816){\color[rgb]{0,0,0}\makebox(0,0)[lt]{\lineheight{1.25}\smash{\begin{tabular}[t]{l}$\beta\equiv 1$\end{tabular}}}}%
	\put(0.71532143,0.21595315){\color[rgb]{0,0,0}\makebox(0,0)[lt]{\lineheight{1.25}\smash{\begin{tabular}[t]{l}$\beta\equiv 0$\end{tabular}}}}%
	\put(0.14780446,0.21918124){\color[rgb]{0,0,0}\makebox(0,0)[lt]{\lineheight{1.25}\smash{\begin{tabular}[t]{l}$\chi\equiv 0$\end{tabular}}}}%
	\put(0,0){\includegraphics[width=\unitlength,page=4]{chi-beta.pdf}}%
	\put(0.14606432,0.07446912){\color[rgb]{0,0,0}\makebox(0,0)[lt]{\lineheight{1.25}\smash{\begin{tabular}[t]{l}$\chi\equiv 1$\end{tabular}}}}%
	\put(0,0){\includegraphics[width=\unitlength,page=5]{chi-beta.pdf}}%
	\put(0.70861965,0.09109867){\color[rgb]{0,0,0}\makebox(0,0)[lt]{\lineheight{1.25}\smash{\begin{tabular}[t]{l}$\beta\equiv 1$\end{tabular}}}}%
	\put(-0.03537575,0.12199424){\color[rgb]{0,0,0}\makebox(0,0)[lt]{\lineheight{1.25}\smash{\begin{tabular}[t]{l}$-\frac{3\varepsilon_k}{4}$\end{tabular}}}}%
	\put(0,0){\includegraphics[width=\unitlength,page=6]{chi-beta.pdf}}%
	\put(0.43989997,0.36979392){\color[rgb]{0,0,0}\makebox(0,0)[lt]{\lineheight{1.25}\smash{\begin{tabular}[t]{l}$S$\end{tabular}}}}%
	\put(0.44391297,0.01343078){\color[rgb]{0,0,0}\makebox(0,0)[lt]{\lineheight{1.25}\smash{\begin{tabular}[t]{l}$S^{-\varepsilon_k}$\end{tabular}}}}%
	\put(0,0){\includegraphics[width=\unitlength,page=7]{chi-beta.pdf}}%
	\put(0.22304309,0.14908759){\color[rgb]{0,0,0}\makebox(0,0)[lt]{\lineheight{1.25}\smash{\begin{tabular}[t]{l}\fontsize{9pt}{1em}transition area for $\chi$\end{tabular}}}}%
	\put(0.54983281,0.14868362){\color[rgb]{0,0,0}\makebox(0,0)[lt]{\lineheight{1.25}\smash{\begin{tabular}[t]{l}\fontsize{9pt}{1em}transition area for $\beta$\end{tabular}}}}%
\end{picture}%
\endgroup%
	\caption{Schematic of the values that the functions $\beta$ and $\chi$ take in $\bb{S\setminus S^{-\eps_k}}+i\tau_k$ for $k\geq 1$.}
	\label{fig:beta}
\end{figure}	

The function $t_k$ is smooth as convolution of a smooth function, $\alpha$, with a continuous function, $\beta$, and it satisfies the following properties:
\begin{enumerate}
\item If $z\nin \bunion k 0 \infty\bb{S^{-\frac{\eps_k}4}\setminus S^{-\frac{3\eps_k}{4}}}$, then $B\bb{z,\frac{\eps_k}{32}}\subset \C\setminus\bb{S^{-\frac{9\eps_k}{32}}\setminus S^{-\frac{23\eps_k}{32}}}$, and so $\beta(z)=1$. Therefore
$$
t_k(z)=\integrate{B(0,\frac{\eps_k}{32})}{}{\frac{32^2}{\eps_k^2}\cdot b\bb{\frac{32w}{\eps_k}}}{w}=\int b(w)dw=1.
$$
\item If $z\in S^{-\frac{3\eps_k}8}\setminus S^{-\frac{5\eps_k}8}$, then $B\bb{z,\frac{\eps_k}{32}}\subset S^{-\frac{11\eps_k}{32}}\setminus S^{-\frac{21\eps_k}{32}}$ and therefore
$$
t_k(z)=\int 0dw=0.
$$
\item If $z\nin \bunion k 0 \infty\bb{S^{-\frac{\eps_k}4}\setminus S^{-\frac{3\eps_k}{4}}},$ then
$$
\frac d{dz} t_k(z)=\int 0 dw,
$$
and if $z\in S^{-\frac{\eps_k}4}\setminus S^{-\frac{3\eps_k}4}$, then $\beta'(z)= \frac{16}{\eps_k}$ and
$$
\abs{\frac d{dz} t_k(z)}\le\int \frac{32^2}{\eps_k^2}\cdot b\bb{\frac{32z}{\eps_k}}\cdot\frac {16}{\eps_k} dw=\frac {16}{\eps_k};
$$
\end{enumerate}
Let $C\defeq 16$ and define the function
$$
\chi(z)\defeq\prodit k 1 \infty t_k(z-i\tau_k)=	\begin{cases}
t_k(z-i\tau_k),& z\in S+i\tau_k \text{ for some }k\geq 0,\\\
1,& \text{otherwise}.
						\end{cases}
$$
This function satisfies (i), (ii), and (iii) as needed.

\end{subproof}

We define the \textit{model map} $h\colon \C \to \C$ by
$$
h(z)\defeq		\begin{cases}
	\chi(z) \bb{h_k(z-i \tau_k)+i \tau_{k+1}},&  z\in S^{-\frac{5\eps_k}8} +i \tau_k,\\
	0,& \chi(z)=0,\\
	\chi(z)\cdot i\tau_0,& \text{otherwise}.
\end{cases}
$$

Then, define the function $g\colon \C \to \C$ as $g(z)\defeq \overline\partial h(z)$, and observe that by the definition of $h$, the properties of $\chi$, and since the maps $h_k$ are holomorphic, 
$$
g(z)
=\begin{cases}
	\overline\partial\chi(z)\cdot \bb{h_k(z-i \tau_k)+i \tau_{k+1}},& z\in \bb{S^{-\frac{5\eps_k}8}\setminus S^{-\frac{3\eps_k}{4}}} +i \tau_k,\\
	\overline\partial\chi(z)\cdot i\tau_0,& z\in \bb{S^{-\frac{\eps_k}4}\setminus S^{-\frac{3\eps_k}{8}}}+i \tau_k,\\
	0,& \text{otherwise}.
\end{cases}
$$

Next, we want to apply \hyperref[thm:Hormander]{Hörmander's Theorem} to $g$ and the subharmonic function $u$ provided by Lemma \ref{lem:sh_stripes}, for which we will show that the integral in the statement of the theorem is finite.

\begin{Claim}There exists a sequence $a_k\nearrow \infty$ with $a_0>10\cdot \tau_0$ such that:
\begin{enumerate}
\item \label{item_1a}For all $z$ with $\Ima(z)<\tau_{k+1}-\frac{1}{2}$, 
\begin{equation}\label{eq_ak}
2\vert h(z)\vert\leq \exp\bb{a_k\exp\bb{\frac{\pi}{\eps_k}\abs{\Rea(z)}}}.
\end{equation}
\item \label{item_3a} For every $k\geq 0$, $a_k$  depends only on  $\{h_j\}^{k+1}_{j=0}$, $\{\eps_j\}^{k+1}_{j=0}$ and $\{\tau_{j}\}^{k+1}_{j=0}$.
\item \label{item_2a} If $u\colon \C\to \R$ is the subharmonic function provided by Lemma~\ref{lem:sh_stripes} for this choice of $\{a_k\}$, then
\begin{equation}\label{eq_bound_int_2}
	\int_\C\abs {g(z)}^2e^{-u(z)}dm(z)<\frac{1}{2}.
\end{equation}
\end{enumerate}
\end{Claim}
\begin{subproof}

Recall that by assumption, for each $k\geq 0$, 
$$\vert h_k(z)\vert^2\le\bb{\vert z\vert^{n_k}+b_k}^2$$
for some $n_k, b_k\in \N$. 
Let
$$P_k(\vert z\vert )\defeq \bb{\vert z\vert^{n_k}+b_k}^2+\tau_0^2+\tau^2_{k+1},$$
therefore, for every $z$ such that $0<\Ima(z)<\tau_{k+1}-\frac{1}{2}$, we can choose $a_k$ such that
\begin{align}\label{eq_align_ak}
	\vert h(z)\vert& \leq P_k(\vert z \vert) \leq P_k\bb{\bb{\bb{\vert \Rea z \vert}^2+\bb{\tau_{k+1}-\frac12}^2}^{\frac{1}{2}}} \nonumber\\ 
	&\leq \frac{1}{2}\exp\bb{a_k\exp\bb{\frac{\pi}{\eps_k}\abs{\Rea(z)}}}.
\end{align}
Since for every $z$ with $\Ima(z)\leq 0$ we have that $h(z)=i\tau_0$ we deduce that \eqref{item_1a} holds for the above choice of $a_k$. 
For our estimates, we will use the polynomial
$$Q_k(x)\defeq C^2\cdot  P_k(x^2+\bb{\tau_k+1}^2)\cdot \bb{x^2+\bb{\tau_k+1}^2}^4,$$ 
where $C>1$ is the constant from property \ref{item:chi_13} of the function $\chi$. Then, by Observation~\ref{obs:int_bnd}, for each $k$, we can choose $a_k>a_{k-1}$ big enough such that  
\begin{equation}\label{eq_Qk}
\integrate 0\infty{Q_k(x)e^{-\frac{a_k}4x}}x<\frac{\eps_k}{2\cdot3^k}
\end{equation}
also holds. Note that $a_k$ only depends on $h_k$, $\tau_k$, $\tau_{k+1}$ and $a_{k-1}$, and so \eqref{item_3a} holds. It is left to show  \eqref{item_2a}.
Let $u\colon \C\to \R$ be the subharmonic function provided by Lemma~\ref{lem:sh_stripes} for this choice of $\{a_k\}$. Then, by property \ref{item:chi_13} of the function $\chi$,
 property \ref{item_u_pasting_2}\ref{item_u_pasting_21} of $u$ and \eqref{eq_Qk},
\begin{align*}
	&\int_\C\abs {g(z)}^2e^{-u(z)}dm(z)\le C^2\sumit k 0 \infty \frac1{\eps_k^2} \int_{S^{-\frac{\eps_k}4}\setminus S^{-\frac{3\eps_k}4}} \abs{h(z+i\tau_k)}^2e^{-u(z+i\tau_k)}dm(z)\\
	&\le C^2\sumit k 0 \infty \frac1{\eps_k^2}\cdot\frac{\eps_k}{2}\integrate{-\infty}{\infty}{P_k(x^2+\bb{\tau_k+1}^2)\cdot e^{-\frac{a_k}{3}\exp\bb{\frac{\pi}{\eps_k} \abs{x}}}\bb{x^2+\bb{\tau_k+1}^2}^4}x\\
	&\le \sumit k 0 \infty \frac1{\eps_k}\integrate 0\infty{Q_k(x)e^{-\frac{a_k}4\cdot \frac{\pi}{\eps_k}x}}x\le \sumit k 0 \infty \frac1{\eps_k}\integrate 0\infty{Q_k(x)e^{-\frac{a_k}4x}}x<\frac12,
\end{align*}
concluding that \eqref{item_2a} holds.
\end{subproof}

Let $\alpha$ be the solution to the $\bar\partial$-equation $\bar\partial\alpha(z)=g(z)$ from \hyperref[thm:Hormander]{Hörmander's Theorem} and define $f\colon \C\to\C$ by
$$f(z)\defeq h(z)-\alpha(z),$$ 
recalling from \eqref{eq_f_entire} that $f$ is entire. 

Let $z\in\C$  
and $r\in\left(0,1\right]$ be such that for all $w\in B\bb{z,r}$, we have $\chi(w)=1$, implying that $h$ is holomorphic in $B(z,r)$ and therefore so is $\alpha=f-h$. Using %
Cauchy's integral formula applied to multiple circles, Cauchy–Schwarz inequality,
\hyperref[thm:Hormander]{Hörmander's Theorem} and \eqref{eq_bound_int_2},
\begin{align*}
	\abs{f(z)-h(z)}&=\frac{1}{\pi r^2}\abs{\integrate{B\bb{z, r}}{}{\alpha(w)}{m(w)}}\le \frac1{\sqrt \pi r}\sqrt{\integrate{B\bb{z, r}}{}{\abs{\alpha(w)}^2}{m(w)}}\\
	&\le \frac1{r}\underset{w\in B\bb{z, r}}\max e^{\frac12\cdot u(w)}\bb{1+\abs w^2}\sqrt{\integrate{B\bb{z, r}}{}{\abs{\alpha(w)}^2\cdot\frac{e^{-u(w)}}{\bb{1+\abs w^2}^2}}{m(w)}}\\
	&\le \frac1{r}\underset{w\in B\bb{z, r}}\max e^{\frac12\cdot u(w)}\bb{1+\abs w^2}\sqrt{\integrate{\C}{}{\abs{g(w)}^2 e^{-u(w)}}{m(w)}}\\
	&\le \frac1{2r}\underset{w\in B\bb{z, r}}\max e^{\frac12\cdot u(w)}\bb{1+\abs w^2}.
\end{align*}
If $z\in (S^{-\eps_k}+i\tau_k)\cup W_k$ for some $k\geq 0$, then, following property \ref{item_u_pasting_2}\ref{item_u_pasting_22} of the subharmonic function $u$, with $r\defeq \eps_k/4$, we have
\[	\abs{f(z)-h(z)}\leq  \frac{2}{\eps_k}\underset{w\in B\bb{z,\eps_k/4}}\max \frac{1+\abs w^2}{\abs w^4}\leq \frac{2+\frac12}{\eps_k\abs z^2}\le \frac{3}{\eps_k\cdot \tau_k^2}=\delta_k,\]
proving properties \eqref{item:pasting_f_hk} and \eqref{item:pasting_f_i} of the function $f$. 

Arguing similarly, with $r=1$, using Cauchy's integral formula, Cauchy–Schwarz inequality, \hyperref[thm:Hormander]{Hörmander's Theorem},  \ref{item_u_pasting_1} of the subharmonic function $u$, and the choice of $a_k$ in \eqref{eq_ak}, the following holds. For every $z$ with $\Ima(z)\le \tau_{k+1}-\frac32$, 
\begin{align*}
	\abs{f(z)}&=\frac{1}\pi\abs{\integrate{B\bb{z,1}}{}{f(w)}{m(w)}}
	\le \frac1{\sqrt \pi}\sqrt{\integrate{B\bb{z, 1}}{}{\abs{h(w)-\alpha(w)}^2}{m(w)}}\\
	&\le 2\underset{w\in B(z,1)}\max\abs{h(w)}+\frac{2}{{\sqrt\pi}}\sqrt{\integrate{B\bb{z,1}}{}{\abs{\alpha(w)}^2}{m(w)}}\\
	&\le 2\underset{w\in B(z,1)}\max\abs{h(w)}+\underset{w\in B\bb{z,1}}\max e^{\frac12\cdot u(w)}\bb{1+\abs w^2}\\
	&\le 2\underset{w\in B(z,1)}\max\abs{h(w)}+3\exp\bb{a_k\exp\bb{\frac{\pi}{\eps_k}\abs{\Rea(z)}}}\le 4\exp\bb{a_k\exp\bb{\frac{\pi}{\eps_k}\abs{\Rea(z)}}}.
\end{align*}
We have thus shown property \eqref{item:pasting_upperbound} of $f$, concluding our proof.
\end{proof}


\section{Functions of small order with unbounded wandering domains}\label{sec_small_order}

The goal of this section is to prove Theorem \ref{thm_order1/2}. In a rough sense, the idea is as follows: given $0<\eps\leq 1/2$, we will choose a sector of angle $\eps$ around $\R_-$, and a sequence of pairwise disjoint disks, $B_n$, of increasingly large radii, tending to infinity. The model map $h$ we use sends the sector to the centre of the first disk,  and each disk to the centre of the next; see Figure \ref{fig:model_1/2}. In Theorem~\ref{thm_small_technical} we construct an entire function $f$ approximating $h$ and of order $\rho(f)=1/2+\eps$.
 In Corollary \ref{cor:sequence}, we show that for the right choice of parameters, this function has a fast escaping wandering domain containing an unbounded domain.

\begin{thm} \label{thm_small_technical} There exists a constant, $C>800$, so that for every $\eps\in\left(0,\frac12\right]$ the following holds. For any $z_1>\exp(\frac{C}{\eps})$, define the sequence $z_n\nearrow \infty$ by
\begin{equation}\label{eq_fast_zn}
z_{n+1}\defeq\frac{\exp\bb{\frac16 z_n^{\frac12+\eps}}}n.
\end{equation}
Moreover, for every $n\geq 1$, define the disk $B_n\defeq B\bb{z_n,\frac{z_n}9}$, and the map
\begin{equation}\label{eq_h}
		h(z)\defeq	\begin{cases}
			z_{n+1},& z\in 3B_n,\\
			z_1,& z\in \bset{\abs{\Arg(z)}>\pi-\eps}^{+1}.
		\end{cases}
\end{equation}
Then there exists an entire function, $f$, with $\rho(f)=\frac12+\eps$ such that
	\begin{enumerate}
		\item  \label{item:f_small_h} For every $z\in \bset{\abs{\Arg(z)}>\pi-\frac\eps 4}^{-\frac14}\cup\bunion 1 n\infty B\bb{z_n,\frac14}$,
		$$
		\abs{f(z)-h(z)}\le\frac{C\cdot z_1\abs {z}^2}{\eps} \exp\bb{-\frac\eps C \abs {z}^{\frac12+\eps}}.
		$$
		\item \label{item:f_small_bound} For every $z\in \C$,
$$
\abs{f(z)}\le 	\frac {C\cdot z_1}{\eps}\exp\bb{\abs z^{\frac12+\eps}}.
$$
	\end{enumerate}
\end{thm}

\begin{figure}[htp]
	\centering
		\def\svgwidth{\linewidth}
\begingroup%
\makeatletter%
\providecommand\color[2][]{%
	\errmessage{(Inkscape) Color is used for the text in Inkscape, but the package 'color.sty' is not loaded}%
	\renewcommand\color[2][]{}%
}%
\providecommand\transparent[1]{%
	\errmessage{(Inkscape) Transparency is used (non-zero) for the text in Inkscape, but the package 'transparent.sty' is not loaded}%
	\renewcommand\transparent[1]{}%
}%
\providecommand\rotatebox[2]{#2}%
\newcommand*\fsize{\dimexpr\f@size pt\relax}%
\newcommand*\lineheight[1]{\fontsize{\fsize}{#1\fsize}\selectfont}%
\ifx\svgwidth\undefined%
\setlength{\unitlength}{745.51181102bp}%
\ifx\svgscale\undefined%
\relax%
\else%
\setlength{\unitlength}{\unitlength * \real{\svgscale}}%
\fi%
\else%
\setlength{\unitlength}{\svgwidth}%
\fi%
\global\let\svgwidth\undefined%
\global\let\svgscale\undefined%
\makeatother%
\begin{picture}(1,0.39923954)%
	\lineheight{1}%
	\setlength\tabcolsep{0pt}%
	\put(0,0){\includegraphics[width=\unitlength,page=1]{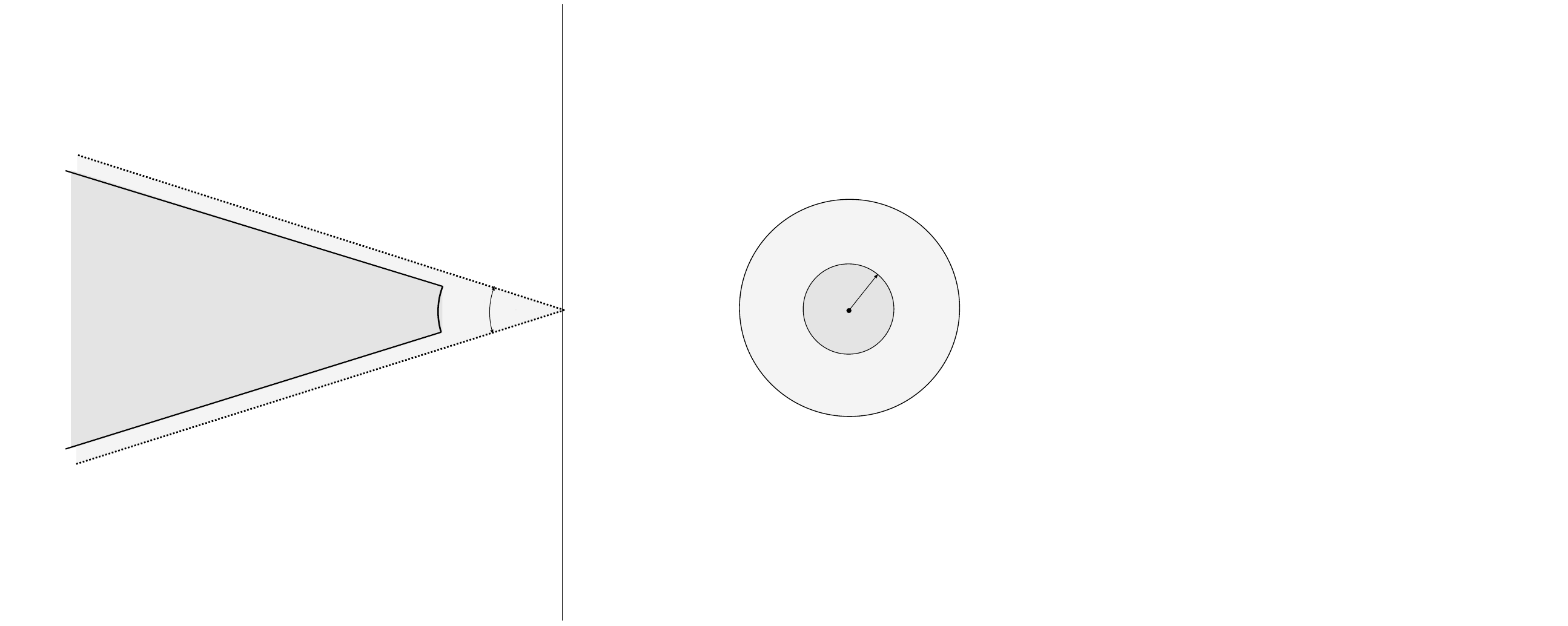}}%
	\put(0.52847129,0.18671775){\color[rgb]{0,0,0}\makebox(0,0)[lt]{\lineheight{1.25}\smash{\begin{tabular}[t]{l}$z_1$\end{tabular}}}}%
	\put(0.52551307,0.28223514){\color[rgb]{0,0,0}\makebox(0,0)[lt]{\lineheight{1.25}\smash{\begin{tabular}[t]{l}$B_1$\end{tabular}}}}%
	\put(0,0){\includegraphics[width=\unitlength,page=2]{domains_section3.pdf}}%
	\put(0.76957756,0.18465944){\color[rgb]{0,0,0}\makebox(0,0)[lt]{\lineheight{1.25}\smash{\begin{tabular}[t]{l}$z_2$\end{tabular}}}}%
	\put(0.76620176,0.31864835){\color[rgb]{0,0,0}\makebox(0,0)[lt]{\lineheight{1.25}\smash{\begin{tabular}[t]{l}$B_2$\end{tabular}}}}%
	\put(0,0){\includegraphics[width=\unitlength,page=3]{domains_section3.pdf}}%
	\put(0.05619426,0.23615884){\color[rgb]{0,0,0}\makebox(0,0)[lt]{\lineheight{1.25}\smash{\begin{tabular}[t]{l}$V$\end{tabular}}}}%
	\put(0.30859071,0.22721775){\color[rgb]{0,0,0}\makebox(0,0)[lt]{\lineheight{1.25}\smash{\begin{tabular}[t]{l}$\fontsize{9pt}{1em}\frac{\varepsilon}{2}$\end{tabular}}}}%
	\put(0,0){\includegraphics[width=\unitlength,page=4]{domains_section3.pdf}}%
	\put(0.3270729,0.28475537){\color[rgb]{0,0,0}\makebox(0,0)[lt]{\lineheight{1.25}\smash{\begin{tabular}[t]{l}$f$\end{tabular}}}}%
	\put(0.63700665,0.25117486){\color[rgb]{0,0,0}\makebox(0,0)[lt]{\lineheight{1.25}\smash{\begin{tabular}[t]{l}$f$\end{tabular}}}}%
	\put(0.80667462,0.22354629){\color[rgb]{0,0,0}\makebox(0,0)[lt]{\lineheight{1.25}\smash{\begin{tabular}[t]{l}$\fontsize{9pt}{1em}\frac{1}{4}$\end{tabular}}}}%
	\put(0.56485543,0.2249664){\color[rgb]{0,0,0}\makebox(0,0)[lt]{\lineheight{1.25}\smash{\begin{tabular}[t]{l}$\fontsize{9pt}{1em}\frac{1}{4}$\end{tabular}}}}%
\end{picture}%
\endgroup%
	\caption{Schematic of the action of the entire function $f$ from Corollary~\ref{cor:sequence}. The unbounded domain $V$ is contained in an unbounded wandering domain $U_0$, and $f^n(V)\subset B(z_n, 1/4)\subset B_n\cap U_n$.}
	\label{fig:model_1/2}
\end{figure}

\begin{cor}\label{cor:sequence}
Let $C>800$ be the constant from Theorem~\ref{thm_small_technical} and let $\eps\in\left(0,\frac12\right]$ be fixed. For every $z_1\ge\exp\bb{\frac{C}\eps}$ the corresponding entire function $f$, constructed in Theorem~\ref{thm_small_technical}, has an unbounded wandering domain, $U_0$, satisfying
$$V\defeq \bset{\vert z\vert >z_1  \text{ and } \abs{\Arg(z)}>\pi-\frac\eps 4}^{-\frac14} \subseteq U\subset A(f),$$
and such that $f^n(V)\subset B(z_n,\frac14)\subset U_n$ for all $n\geq 1$.
\end{cor}

\begin{proof}[Proof of Theorem \ref{thm_order1/2}]
It is a direct consequence of Theorem \ref{thm_small_technical} and Corollary \ref{cor:sequence}.
\end{proof}

\begin{proof}[Proof of Corollary \ref{cor:sequence}, using Theorem \ref{thm_small_technical}]
The function $t\mapsto t^2\exp\bb{-\frac\eps C\sqrt t}$ is strictly decreasing on $\left[z_1,\infty\right)\subset\left[\frac{16C^2}{\eps^2},\infty\right)$, since $\frac{16C^2}{\eps^2}\leq \exp\bb{\frac{C}\eps}<z_1$, and satisfies that for every $t>z_1$
$$
\frac{C\cdot z_1}\eps\cdot t^2\exp\bb{-\frac\eps C\sqrt t}<\frac{C\cdot z_1^3}\eps\cdot\exp\bb{-\frac\eps C\sqrt {z_1}}<\frac14.
$$
Let $f$ be the function constructed in Theorem \ref{thm_small_technical} for this choice of parameters. Then, for every $z\in V$,
\begin{align*}
\abs{f(z)-h(z)}=\abs{f(z)-z_1}&\le \frac{C\cdot z_1\abs {z}^2}{\eps} \exp\bb{-\frac\eps C\abs {z}^{\frac12+\eps}}\\
&\le \frac{C\cdot z_1\abs {z}^2}{\eps} \exp\bb{-\frac\eps C \sqrt{\abs z}}<\frac14,
\end{align*}
following property \eqref{item:f_small_h} of $f$ and the lower bound imposed on $z_1$. We conclude that $f(V)\subset B\bb{z_1,\frac14}$. Similarly, for every $n\ge 1$ we have
\begin{equation}\label{eq_fB}
f(B(z_{n-1}, 1/4))\subset B(z_{n},1/4).
\end{equation}
In particular, for every $z\in V$ the iterates $\bset{f^n(z)}_{n=1}^\infty\subset\bset{\Rea(z)>0}$. Hence, by Montel's theorem, $V$ is contained in an unbounded Fatou component, $U_0$, of $f$, and for every $n\geq 1$, $B(z_{n},\frac{1}{4})$ is contained in a Fatou component $U_n$ of $f$. %

Next, let 
$$
N\defeq \min\left\{n\geq 100 \colon \frac{z_n}{100}>\min_{z\in J(f)}\vert z\vert\text{ and } \frac {C\cdot z_1}{\eps}\exp\bb{\frac{ z_{k}^{\frac12+\eps}}{10}}<\frac{\exp\bb{\frac16 z_{k}^{\frac12+\eps}}}{100 {k}} \text{ for all }k\ge n\right\},
$$
and define $R_0 \defeq \frac{z_N}{100}$, noting that, as in the proof of Corollary \ref{cor_wd_pasting}, for every $r>R_0$ we have $M_f(r)>r$. Let $R_n\defeq M_f(R_{n-1})$. We will show by induction that 
\begin{equation*}
	R_{n}\le \frac{z_{n+N}}{100}.
\end{equation*}
Note that this holds for $n=0$. Assume it holds for $R_{n-1}$. Then, following property \eqref{item:f_small_bound} of $f$ and the choice of the sequence $\{z_n\}$,
\begin{align*}
	R_n&= M_f(R_{n-1})\le M_f\bb{\frac{z_{n-1+N}}{100}}\le \frac {C\cdot z_1}{\eps}\exp\bb{\bb{\frac{z_{n-1+N}}{100}}^{\frac12+\eps}}\\
	&\le  \frac {C\cdot z_1}{\eps}\exp\bb{\frac{z_{n-1+N}^{\frac12+\eps}}{10}}\le \frac{\exp\bb{\frac16z_{n-1+N}^{\frac12+\eps}}}{100(n-1+N)}=\frac{z_{n+N}}{100}.
\end{align*}

Lastly, for every $z\in V$ and $n\geq N$,
$$
\abs{f^{n}(z)}>z_{n}-\frac14>\frac{z_n}{100}\ge R_{n-N}, 
$$
and so $V\subset A(f)$.

We conclude that $U_n\subset A(f)$ is a wandering domain for all $n\geq 0$, see e.g. \cite[Corollary 4.2]{PG_fastescaping_12}. In particular, $U \defeq U_0$ is an unbounded wandering domain, and $f^n(V)\subset B(z_n,\frac14)\subset U_n$ for all $n\geq 1$, as claimed.
\end{proof}

As in the proof of Theorem \ref{thm:pasting_strips}, we will approximate our model map $h$ using \hyperref[thm:Hormander]{Hörmander's Theorem}. This time, the construction of the subharmonic function $u$ we shall use, will focus on keeping the order small, and therefore different tools will be used. See Figure \ref{fig:sub_sec3} for a schematic of the values that $u$ takes.

\begin{lemma}\label{lem:sh_power1/2} There exists a constant $c_0>4$ so that for every $\eps\in\left(0,\frac12\right]$ and sequences $\bset{z_n}, \bset{B_n}$ as described in Theorem  \ref{thm_small_technical}, there exists a subharmonic function, $u\colon \C \to \R$, such that:
	\begin{enumerate}[label=(\arabic*)]
		\item \label{item:u_small_1}For all $z\in \C$,
		$$
		u(z)\le\abs z^{\frac12+\eps};
		$$
		\item \label{item:u_small_2} For every $n\in \N$,
		\begin{enumerate}[label=$\bullet$]
		\item$u|_{B(z_n,1)}\le 0$.
		\item  $u|_{B\bb{z_n,\frac12}}\le -c\abs {z_n}^{\frac12+\eps}$ for some uniform constant $c\in\bb{0,1}$.
		\end{enumerate}
		\item \label{item:u_small_3} 
		\begin{enumerate}[label=$\bullet$]
		\item If $\abs{\Arg(z)}>\pi-\frac\eps 2$, then $u(z)\leq 0$.
		\item If $\abs{\Arg(z)}>\pi-\frac\eps 4$, then
		$
		u(z)\le -\frac\eps 8\abs z^{\frac12+\eps};
		$
		\end{enumerate}
		\item \label{item:u_small_4}
		\begin{enumerate}[label=$\bullet$]
		\item If $\abs{\Arg(z)}\in\bb{\frac\pi2,\pi-\eps}$, then $u(z)\ge \frac\eps 4\abs z^{\frac12+\eps}$.
		\item If $\abs{\Arg(z)}<\frac\pi 3$ and $z\nin \bunion n 1 \infty B_n$, then $u(z)\ge \frac12\abs z^{\frac12+\eps}$.
		\end{enumerate}
	\end{enumerate}
	
\end{lemma}
\begin{proof}
	The plan of this construction is to create a subharmonic function that is at most zero on the domain $\bset{z,\; \abs{\Arg(z)}>\pi-\frac\eps 2}$ and then use Poisson integrals to create `puddles', i.e., open disks where the function $u$ is very negative,  in the disks $B_n$. 
	
	For the rest of the proof, let us fix $\eps\in\left(0,\frac12\right]$.	We begin with constructing a subharmonic function, $v$, satisfying properties,\ref{item:u_small_1}, \ref{item:u_small_3} and~\ref{item:u_small_4}.
	Let $$
	v(z)\defeq \abs z^\alpha\cos\bb{\beta\cdot \Arg(z)}, \quad \text{for } \alpha\defeq \frac12+\eps \quad \text{ and } \quad  \beta\defeq \frac12+\frac\eps{4\pi-2\eps}.
	$$
Note that 	$$\frac\pi{2\beta}=\frac\pi{2\bb{\frac12+\frac\eps{4\pi-2\eps}}}=\frac\pi{1+\frac\eps{2\pi-\eps}}=\pi-\frac\eps2.
	$$
	Therefore, if $z$ is such that $\abs{\Arg(z)}\ge\frac\pi{2\beta}=\pi-\frac\eps2$, then $\cos(\beta\cdot \Arg(z))<0$, and so the first part of property \ref{item:u_small_3} holds for $v$. 
	To see the second part of this property, let $z$ be such that $\abs{\Arg(z)}\ge\pi-\frac\eps4$.  Then, for every $\eps \in \left(0,\frac12\right]$
	$$
	\cos\bb{\beta\cdot \Arg(z)}\le \cos\bb{\beta\cdot\bb{\pi-\frac\eps 4}}<-\frac\eps 8,
	$$
	implying that
	$$
	v(z)\le -\frac\eps8\abs z^{\frac12+\eps}.
	$$
	Similarly, if $z$ is such that  $\abs{\Arg(z)}\in\bb{\frac\pi 2, \pi-\eps}$, then for every  $\eps\in\left(0,\frac12\right]$
	$$
	\cos\bb{\beta \Arg(z)}\ge\cos\bb{\beta\bb{\pi-\eps}}\ge \frac\eps 4, 
	$$
	and so
	$$
	v(z)\ge \frac\eps4\abs z^{\frac12+\eps},
	$$
	implying the first part of property \ref{item:u_small_4}. For the second part note that if $\abs{\Arg(z)}<\frac\pi 3$, then since cosine is monotone decreasing,
	$$
	\cos\bb{\beta\cdot \Arg(z)}\ge\cos\bb{\frac\pi 3}=\frac12,
	$$
	implying that $$
	v(z)=\abs z^\alpha\cos\bb{\beta\cdot \Arg(z)}\ge \frac12\abs z^{\frac12+\eps},
	$$ as required. 
	Lastly, 
	$$
	\Delta v(z)=\abs z^{\alpha-2}\cos\bb{\beta \Arg(z)}\bb{\alpha\bb{\alpha-1}+\alpha-\beta^2}=\abs z^{\alpha-2}\cos\bb{\beta \Arg(z)}\bb{\alpha^2-\beta^2}>0,
	$$
 as $\alpha>\beta$  and so $v$ is a subharmonic function.
 
Next, we will attach the `puddles' in $B_n$ using Poisson integral formula (see, for example, \cite[Definition 1.2.3]{ransford}). For every $n\geq 1$, let $v_n(z)\defeq v\bb{r_n\cdot z+z_n}$, where $r_n\defeq z_n/9$ is the radius of $B_n$, so that $v_n\bb{\D}=v(B_n)$, and define
	$$
	u(z)\defeq  \begin{cases}
		P_\D\bb{v_n}\bb{\frac{z-z_n}{r_n}}+A_n\cdot \log\bb{\frac{\abs{z-z_n}}{r_n}},& z\in B_n,\\
		v(z),& \text{otherwise},
	\end{cases}
	$$
	where $A_n>0$ are constants that will be chosen momentarily, and $P_\D(v_n)$ is the Poisson integral of $v_n$. 
	It is immediate to see that properties \ref{item:u_small_3} and \ref{item:u_small_4} hold for $u$. Also since \ref{item:u_small_1} holds for $v$, by the maximum modulus principle, it also holds for $u$. It remains to prove property \ref{item:u_small_2} and that $u$ is subharmonic. Note that, by definition, for every $n$, $u(z_n)=-\infty$ and, therefore, if $A_n$ is large enough, then property \ref{item:u_small_2} holds. On the other hand, in order for the function $u$ to be subharmonic, we will see that $A_n$ cannot be too large. The rest of the proof is dedicated to proving that $A_n$ can be chosen to satisfy both requirements.
	
	\textbf{ Lower bound}: Note that for the first part of property \ref{item:u_small_2} to hold for $u$, $A_n$ needs to be chosen large enough so that 
	\begin{equation}\label{eq_An_1}
			\begin{split}
			\underset{\abs{z-z_n}=1}\max\; \left(P_\D\bb{v_n}\bb{\frac{z-z_n}{r_n}}+A_n\log\bb{\frac{\abs{z-z_n}}{r_n}}\right)&\le \underset{z\in B_n}\max\; v(z)-A_n\log(r_n)\\
			&\le\bb{\abs{z_n}+r_n}^{\frac12+\eps}-A_n\log(r_n)<0\\
			&\iff A_n\ge \frac{\bb{\abs{z_n}+r_n}^{\frac12+\eps}}{\log(r_n)}.
		\end{split}
	\end{equation}

	\textbf{ Upper bound}: To preserve subharmonicity, we will use the following claim.
	\begin{Claim}\label{obs:glueing}
		Let $\Omega\subseteq\C$ be a domain, and let $\Omega_1\subset\subset\Omega$ be a domain, so that $\partial\Omega_1$ is an orientable smooth curve. Every function $u$ which is continuous on $\Omega$ and subharmonic on $\Omega_1\cup \Omega\setminus\overline{\Omega_1}$, is subharmonic on $\Omega$ if on $\partial\Omega_1$ it satisfies
		$$
		\frac{\partial u}{\partial n_1}\le \frac{\partial u}{\partial n_2},
		$$
		where $n_1$ is the outer normal to $\Omega_1$ along $\partial\Omega_1$ and $n_2$ is the outer normal to $\Omega\setminus\Omega_1$ along~$\partial\Omega_1$. 
	\end{Claim}
\begin{subproof*}
One can show that $\Delta u \geq 0$ on $\Omega$ as a measure. This can be deduced by using Green's second identity (\cite[Theorem~1.9]{hayman_subharmonic}) to show that the integral of $u \Delta \phi$ with respect to the Lebesgue measure is non-negative for any smooth function $\phi$ compactly supported on a neighbourhood of $\Omega \setminus \overline{\Omega_1}$.
 \end{subproof*} 

 We will use it with $\Omega\defeq B\bb{z_n, r_n +\frac12}$, and $\Omega_1\defeq  B_n$. In this case, $\frac{\partial u}{\partial n_2}=\frac{\partial v}{\partial n}$. To calculate $\frac{\partial u}{\partial n_1}$ we will use Poisson-Jensen's formula (see, for example, \cite[Theorem~4.5.1]{ransford}). The problem is then reduced to showing that for every $\xi\in \partial\D$:

	\begin{align*}
		\frac{\partial u}{\partial n_1}(z_n+r_n\xi)&=\limit r {1^-}\frac{v\bb{z_n+ r_n\xi}-P_{\D}v_n(r \xi)}{1-r}+A_n=r_n \frac{\partial v}{\partial n}\bb{z_n+ r_n\xi}+\limit r {1^-}\frac{G_{\D}\bb{r \xi}}{1-r}+A_n\\
		&= r_n\frac{\partial v}{\partial n}\bb{z_n+ r_n\xi}-\frac{\partial G_{\D}}{\partial r}\bb{\xi}+A_n\le r_n\frac{\partial u}{\partial n_2}(z_n+r_n\xi)=r_n\frac{\partial v}{\partial n}\bb{z_n+ r_n\xi},
	\end{align*}
	where if $g_\D$ denotes Green's function for the unit disk, $\D$, then
	$$
	G_{\D}(\xi)\defeq \frac1{2\pi}\integrate {\D}{}{g_{\D}(\xi,y)\Delta v_n(y)}m(y).
	$$
	We conclude that in order to show subharmonicity of $u$, it is enough to restrict $A_n$ so that for every $\xi\in \partial\D$,
	\begin{equation}\label{eq_An_2}
	A_n\le \frac{\partial G_{\D}}{\partial r}\bb{\xi}.
	\end{equation}
	Let us begin by finding a lower bound for $\frac{\partial G_{\D}}{\partial r}\bb{\xi}$ to get an upper bound on $A_n$. Because the collection $\bset{\frac{\partial g_{\D}}{\partial n}\bb{\xi,\cdot}}_{\xi\in\partial\D}$ is uniformly integrable, we may switch between the integral and the derivative. By the way we chose $\alpha$ and $\beta$, for every $n\in \N$ and $y\in\D$,
	$$
	\Delta v_n(y)=r_n^2\Delta v(z_n+r_n\cdot y)=r_n^2\cdot \frac{v(z_n+r_n\cdot y)}{\abs {z_n+r_n\cdot y}^2}\bb{\alpha^2-\beta^2}>\frac{v(z_n+r_n\cdot y)}{\abs {\frac{z_n}{r_n}+ y}^2}\cdot\frac\eps 2.
	$$
	We see that
	\begin{eqnarray*}
		\frac{\partial G_{\D}}{\partial r}\bb{\xi}&=&\frac{\partial}{\partial r}\bb{\frac1{\pi}\integrate {\D}{}{g_{\D}(\xi,y)\Delta v_n(y)}m(y)}=\frac1{2\pi}\integrate {\D}{}{\frac{\partial g_{\D}(\cdot,y)}{\partial r}(\xi)\Delta v_n(y)}m(y)\\
		&\ge&\underset{y\in\D}\inf\; \Delta v_n(y)\frac1{2\pi}\integrate {\D}{}{\frac{\partial g_{\D}(\cdot,y)}{\partial r}(\xi)}m(y)\ge c\cdot\eps\cdot \bb{\abs {z_n}-r_n}^{\frac12+\eps}
	\end{eqnarray*}
	for some numerical constant $c\in (0,1)$, where we have used property \ref{item:u_small_4}. Thus, for $A_n$ to satisfy both \eqref{eq_An_1} and \eqref{eq_An_2}, we need to choose

	\begin{equation}\label{eq_An}
	\frac{\bb{\abs{z_n}+r_n}^{\frac12+\eps}}{\log(r_n)}\le A_n\le c\cdot\eps\cdot \bb{\abs {z_n}-r_n}^{\frac12+\eps}.
	\end{equation}
Note that
	$$
	\frac{\bb{\abs{z_n}+r_n}^{\frac12+\eps}}{\log(r_n)}\le c\eps\bb{\abs {z_n}-r_n}^{\frac12+\eps}\iff \log(r_n)>\bb{\frac{\abs{z_n}+r_n}{\abs{z_n}-r_n}}^{\frac12+\eps}\frac1{c\eps}=\bb{\frac54}^{\frac12+\eps}\frac1{c\eps},
	$$
	and the inequality on the right holds if $z_1> \exp(\frac{c_0}{\eps})$ for some numerical constant $c_0$ that can be chosen bigger than $4$. Thus, \eqref{eq_An} holds for $A_n\defeq c\cdot\eps\cdot \bb{\abs {z_n}-r_n}^{\frac12+\eps}$.
	
We are only left to show the second part of property \ref{item:u_small_2}. For $\abs{z-z_n}\le\frac12$, we have
		\begin{align*}
		u(z)& = P_\D\bb{v_n}\bb{\frac{z-z_n}{r_n}}+A_n\log\bb{\frac{\abs{z-z_n}}{r_n}}\\
		&\le P_\D\bb{v_n}\bb{\frac{z-z_n}{r_n}}+A_n\log\bb{\frac{1}{r_n}}-A_n\log(2)\le -A_n\log(2),
	\end{align*}
	concluding the proof.
\end{proof}

\begin{figure}[htp]
	\centering
	\def\svgwidth{\linewidth}
\begingroup%
\makeatletter%
\providecommand\color[2][]{%
	\errmessage{(Inkscape) Color is used for the text in Inkscape, but the package 'color.sty' is not loaded}%
	\renewcommand\color[2][]{}%
}%
\providecommand\transparent[1]{%
	\errmessage{(Inkscape) Transparency is used (non-zero) for the text in Inkscape, but the package 'transparent.sty' is not loaded}%
	\renewcommand\transparent[1]{}%
}%
\providecommand\rotatebox[2]{#2}%
\newcommand*\fsize{\dimexpr\f@size pt\relax}%
\newcommand*\lineheight[1]{\fontsize{\fsize}{#1\fsize}\selectfont}%
\ifx\svgwidth\undefined%
\setlength{\unitlength}{762.51968504bp}%
\ifx\svgscale\undefined%
\relax%
\else%
\setlength{\unitlength}{\unitlength * \real{\svgscale}}%
\fi%
\else%
\setlength{\unitlength}{\svgwidth}%
\fi%
\global\let\svgwidth\undefined%
\global\let\svgscale\undefined%
\makeatother%
\begin{picture}(1,0.68773234)%
	\lineheight{1}%
	\setlength\tabcolsep{0pt}%
	\put(0,0){\includegraphics[width=\unitlength,page=1]{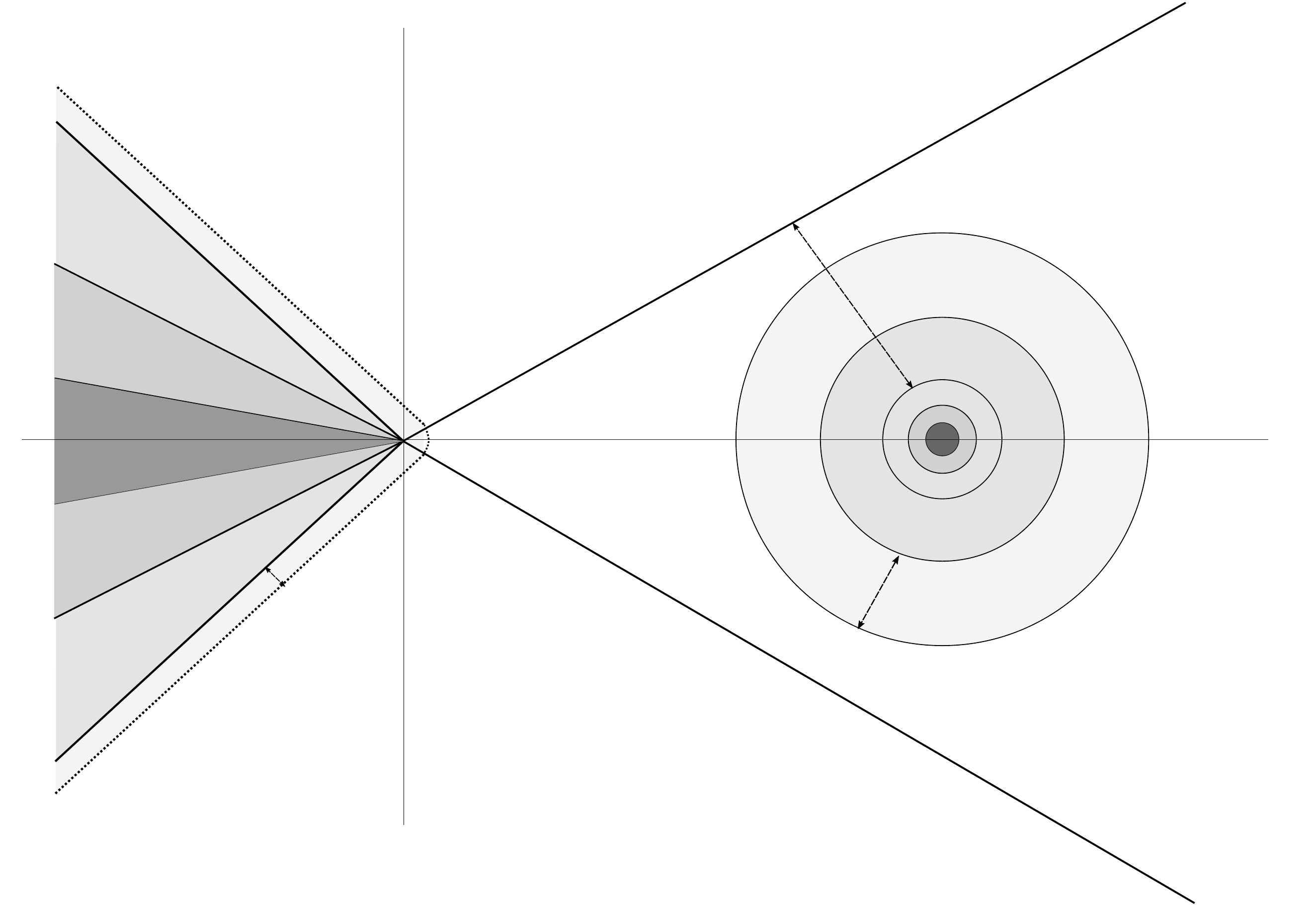}}%
	\put(0.78869126,0.41435206){\color[rgb]{0,0,0}\makebox(0,0)[lt]{\lineheight{1.25}\smash{\begin{tabular}[t]{l}$\fontsize{9pt}{1em}2B_n$\end{tabular}}}}%
	\put(0.75348324,0.38188213){\color[rgb]{0,0,0}\makebox(0,0)[lt]{\lineheight{1.25}\smash{\begin{tabular}[t]{l}$\fontsize{9pt}{1em}B_n$\end{tabular}}}}%
	\put(0.04704144,0.36164922){\color[rgb]{0,0,0}\makebox(0,0)[lt]{\lineheight{1.25}\smash{\begin{tabular}[t]{l}$\fontsize{9pt}{1em}u(z)\!\ll\! 0$\end{tabular}}}}%
	\put(0.61228631,0.51260299){\color[rgb]{0,0,0}\makebox(0,0)[lt]{\lineheight{1.25}\smash{\begin{tabular}[t]{l}$\fontsize{9pt}{1em}u(z)\!\gg\!0$\end{tabular}}}}%
	\put(0.0422522,0.49224972){\color[rgb]{0,0,0}\makebox(0,0)[lt]{\lineheight{1.25}\smash{\begin{tabular}[t]{l}$\fontsize{9pt}{1em}u(z)\!\gg\! 0$\end{tabular}}}}%
	\put(0.04573305,0.41867699){\color[rgb]{0,0,0}\makebox(0,0)[lt]{\lineheight{1.25}\smash{\begin{tabular}[t]{l}$\fontsize{9pt}{1em}u(z)\!\leq \!0$\end{tabular}}}}%
	\put(0.04961326,0.19569476){\color[rgb]{0,0,0}\makebox(0,0)[lt]{\lineheight{1.25}\smash{\begin{tabular}[t]{l}$\fontsize{9pt}{1em}\chi \equiv 1$\end{tabular}}}}%
	\put(0.68825185,0.28575242){\color[rgb]{0,0,0}\makebox(0,0)[lt]{\lineheight{1.25}\smash{\begin{tabular}[t]{l}$\fontsize{9pt}{1em}\chi \equiv 1$\end{tabular}}}}%
	\put(0.32392512,0.10443024){\color[rgb]{0,0,0}\makebox(0,0)[lt]{\lineheight{1.25}\smash{\begin{tabular}[t]{l}$\fontsize{9pt}{1em}\chi \equiv 0$\end{tabular}}}}%
	\put(0.68892629,0.22391612){\color[rgb]{0,0,0}\makebox(0,0)[lt]{\lineheight{1.25}\smash{\begin{tabular}[t]{l}$\fontsize{9pt}{1em}\nabla \chi \neq 0$\end{tabular}}}}%
	\put(0.21047015,0.25235419){\color[rgb]{0,0,0}\makebox(0,0)[lt]{\lineheight{1.25}\smash{\begin{tabular}[t]{l}$\fontsize{9pt}{1em}\nabla \chi \neq 0$\end{tabular}}}}%
	\put(0.41310063,0.38820447){\color[rgb]{0,0,0}\makebox(0,0)[lt]{\lineheight{1.25}\smash{\begin{tabular}[t]{l}$\frac{2\pi}{3}$\end{tabular}}}}%
	\put(0.83293215,0.46036089){\color[rgb]{0,0,0}\makebox(0,0)[lt]{\lineheight{1.25}\smash{\begin{tabular}[t]{l}$\fontsize{9pt}{1em}3B_n$\end{tabular}}}}%
	\put(0,0){\includegraphics[width=\unitlength,page=2]{u_xi_section3.pdf}}%
\end{picture}%
\endgroup%
	\caption{Schematic of the values that the subharmonic function $u$ from Lemma \ref{lem:sh_power1/2} and the function $\chi$ from the proof of Theorem \ref{thm_small_technical} take.}
	\label{fig:sub_sec3}
\end{figure}

\begin{proof}[Proof of Theorem \ref{thm_small_technical}] Fix $\eps\in (0, 1/2]$ and let $c_0>4$ be the numerical constant from Lemma \ref{lem:sh_power1/2}. For any $z_1> \exp(\frac{c_0}{\eps})$, let $\bset{z_n}\subset \C$ be a collection of points as in the statement of the theorem, with the corresponding disks $B_n\defeq B(z_n, r_n)$ for $r_n\defeq \frac{z_n}9$.
	
Note that a simple calculation shows that for our choice of $z_1$, the sequence $\{z_n\}$ is monotone increasing, and for every $n$,
$$
3B_n\cap 3B_{n-1}=\emptyset
$$
since otherwise,
$$
z_n-z_{n-1}<3\bb{r_n+r_{n-1}}<6\cdot r_n=\frac23 z_n,
$$
which forms a contradiction as $z_n\gg z_{n-1}$.

Define
$$T\defeq \bset{\abs{\Arg(z)}>\pi-\eps}^{+1}\setminus \bset{\abs{\Arg(z)}>\pi-\eps}.$$
\begin{Claim}We can choose a smooth map  $\chi\colon \C\to [0,1]$ such that:
\begin{enumerate}[label=(\roman*)]
		\item For every $z\in \bunion n 1 \infty 2B_n \cup \bset{\abs{\Arg(z)}>\pi-\eps}$, $\chi(z)=1$.
		\item For every $z\in \C\setminus\bb{\bunion n 1 \infty 3B_n \cup \bset{\abs{\Arg(z)}>\pi-\eps}^{+1}}$, $\chi(z)=0.$
		\item \label{item:chi_c}  $\nabla\chi$ is supported on 
		$
		T\cup\bunion n 1 \infty \bb{3B_n\setminus 2B_n}
		$
		and there exists a numerical constant $\tilde{C}>1$ so that
		$$
		\abs{\nabla\chi(z)}\le	\begin{cases}
			\tilde{C},& z\in T,\\
			\frac{\tilde{C}}{z_n},& z\in \bunion n 1 \infty (3B_n\setminus 2B_n).
		\end{cases}
		$$
	\end{enumerate}
\end{Claim}

\begin{subproof}
To prove the existence of such a smooth map, one repeats the argument presented in the proof of Theorem \ref{thm:pasting_strips}, with $\eps_k=r_n$ and defines $t_k$, for $k\geq 1$, using convolution of the function $b(z)$, see \eqref{eq_bump}, with
	$$\beta(z)\defeq
	\begin{cases}
		1,& z\in \bunion n 1 \infty B\bb{z_n,2r_n+1},\\
		0,& z\nin \bunion n 1 \infty B\bb{z_n,3r_n-1},\\
		\frac1{r_n-2}\bb{3r_n-1-\abs{z-z_n}},& z\in \bunion n 1 \infty (B\bb{z_n,3r_n-1}\setminus B\bb{z_n,2r_n+1}).
	\end{cases}
	$$
	A map $t_0$ can be constructed in a similar way to accommodate the set $T$ and define the desired function as $\chi(z)\defeq \sumit k 0\infty t_k(z)$.
\end{subproof}

Define $g\colon \C\to\C$ as $g(z)\defeq \bar\partial\bb{\chi(z)\cdot h(z)}$, where the map $h$ is defined in \eqref{eq_h}, noting that $g(z)=\bar\partial\chi(z)\cdot h(z).$ Moreover, let $u\colon \C\to \R$ be the subharmonic function provided by Lemma~\ref{lem:sh_power1/2}. We shall obtain our entire function $f$ by applying \hyperref[thm:Hormander]{Hörmander's Theorem} to the functions $u$ and $g$. For that, we first require to show finiteness of the integral in the statement of the theorem.

\begin{Claim} For the functions $g$ and $u$ specified above, we have
	\begin{equation}\label{eq_bound_int}
		\int_\C\abs {g(z)}^2e^{-u(z)}dm(z)< \bb{\frac{10\tilde C z_1}{\eps}}^2,
	\end{equation}	
	where $\tilde C>1$ is the constant from property \ref{item:chi_c} of the function $\chi$.
\end{Claim}
\begin{subproof}
Following property \ref{item:chi_c}  of $\chi$ and \eqref{eq_h} we see that
$$
\int_\C\abs {g(z)}^2e^{-u(z)}dm(z) \le \tilde{C}^2\bb{\int_{T}   z_1^2e^{-u(z)}dm(z)+\sumit n 1 \infty \frac{z_{n+1}^2}{z_n^2}\integrate{3B_n\setminus 2B_n}{}{e^{-u(z)}}{m(z)}}.
$$
We will begin by bounding the integral on each disk,  $B_n$. Note that if $z\in 3B_n$, then
$$
\abs{\Arg(z)}\le \arctan\bb{\frac{3r_n}{z_n-3r_n}}=\arctan\bb{\frac12}<\frac\pi 3.
$$
Using the second part of property \ref{item:u_small_4} of $u$, 
\begin{align*}
\integrate{3B_n\setminus 2B_n}{}{e^{-u(z)}}{m(z)}&\le m(3B_n\setminus 2B_n)\cdot\underset{z\in 3B_n\setminus 2B_n}\sup\; e^{-u(z)}\le 5\pi\cdot r_n^2\cdot e^{-\frac12\underset{z\in 3B_n\setminus 2B_n}\inf\; \abs z^{\frac12+\eps}}\\
&\le 5\pi\cdot r_n^2\cdot e^{-\frac12\bb{z_n-3r_n}^{\frac12+\eps}}\le  5\pi\cdot r_n^2\cdot e^{-\frac13\cdot z_n^{\frac12+\eps}}=\frac{5\pi\cdot r_n^2}{z_{n+1}^2\cdot n^2},
\end{align*}
by the way the sequence $\bset{z_n}$ is defined.

Next, along the sides of the sector, on the set $T$, we will use the first part of property \ref{item:u_small_4} of $u$, 
\begin{align*}
\int_T   z_1^2e^{-u(z)}dm(z)&=z_1^2\cdot 2\integrate 0 \infty{\exp\bb{-\frac\eps 4\cdot  t^{\frac12+\eps}}}t\le \frac{64z_1^2}{\eps^2}.
\end{align*}
Combining both estimates together we see that,
\begin{align*}
&\int_\C\abs {g(z)}^2e^{-u(z)}dm(z) \le \tilde{C}^2\bb{\int_T  z_1^2e^{-u(z)}dm(z)+\sumit n 1 \infty \frac{z_{n+1}^2}{z_n^2}\integrate{3B_n\setminus 2B_n}{}{e^{-u(z)}}{m(z)}}\\
&\le \tilde{C}^2\bb{\frac{64z_1^2}{\eps^2}+ \sumit n 1\infty\frac{z_{n+1}^2}{z_n^2}\cdot \frac{5\pi\cdot r_n^2}{z_{n+1}^2\cdot n^2}}= \tilde{C}^2\bb{\frac{64z_1^2}{\eps^2}+ \sumit n 1\infty\frac{5\pi}{81\cdot n^2}}\le 100\tilde C^2\cdot\frac{z_1^2}{\eps^2},
\end{align*}
concluding the proof.
\end{subproof}

By \hyperref[thm:Hormander]{Hörmander's Theorem}, there exists a smooth function $\alpha$ so that $\bar\partial\alpha(z)=g(z)$ and (\ref{eq:Hormander}) holds. Define $f\colon \C\to \C$ as 
$$f(z)\defeq \chi(z)\cdot h(z)-\alpha(z), $$
and note that arguing as in \eqref{eq_f_entire}, it follows that $f$ is entire.

To prove the properties of $f$ we argue similarly as in the proof of Theorem \ref{thm:pasting_strips}. For the rest of the proof, let us fix the constant
\begin{equation}\label{eq_C}
C\defeq 800 \left(\tilde{C}+\frac{1}{c}+c_0\right)>800,
\end{equation}
where $c\in (0,1)$ is the constant from Lemma \ref{lem:sh_power1/2}\ref{item:u_small_2}.

To show property \eqref{item:f_small_h} of $f$, let $z\in B\bb{z_n,\frac14}\cup \bset{\abs{\Arg(z)}>\pi-\frac\eps 4}^{-\frac14}$, and observe that $B\bb{z,\frac14}\subset B\bb{z_n,\frac12}\cup \bset{\abs{\Arg(z)}>\pi-\frac\eps 4}$. Then, by properties \ref{item:u_small_2} and \ref{item:u_small_3} of the subharmonic function $u$,
	$$
	u|_{B\bb{z,\frac14}}(w)\le 	\begin{cases}
		-\frac{c}{2}\abs {z_n}^{\frac12+\eps},& w\in B\bb{z_n,\frac12},\\
		-\frac\eps 8\abs z^{\frac12+\eps},& w\in\bset{\abs{\Arg(z)}>\pi-\frac\eps 4}, 
	\end{cases}
	$$
	while $\chi|_{B\bb{z,\frac14}}\equiv 1$, i.e., $\alpha$ is holomorphic in $B\bb{z,\frac14}$. In particular, 
		\begin{equation}\label{eq_max_ball}
		\underset{w\in B\bb{z,\frac14}}\max\; e^{\frac12\cdot u(w)}  \le \exp\bb{-\frac\eps C\abs z^{\frac12+\eps}}.
		\end{equation}
	Using Cauchy's integral formula, Cauchy–Schwarz inequality, \hyperref[thm:Hormander]{Hörmander's Theorem}, and equations \eqref{eq_bound_int}, \eqref{eq_max_ball} and \eqref{eq_C}, we have that
	\begin{align*}
		\abs{f(z)-h(z)}&=\abs{\alpha(z)}=\frac{16}\pi\abs{\integrate{B\bb{z,\frac14}}{}{\alpha(w)}{m(w)}}\le 3\sqrt{\integrate{B\bb{z,\frac14}}{}{\abs{\alpha(w)}^2}{m(w)}}\\
		&\le 3\underset{w\in B\bb{z,\frac14}}\max\; e^{\frac12\cdot u(w)}\cdot\bb{1+\abs w^2}\sqrt{\integrate{\C}{}{\abs{\alpha(w)}^2\frac{e^{-u(w)}}{\bb{1+\abs w^2}^2}}{m(w)}}\\
		&\le 10\cdot \abs z^2 \underset{w\in B\bb{z,\frac14}}\max\; e^{\frac12\cdot u(w)}  \cdot \frac{10\tilde C z_1}{\eps}\le\frac{C\cdot z_1\abs z^2}\eps\exp\bb{-\frac\eps C\abs z^{\frac12+\eps}},	
	\end{align*}
	concluding the proof of \eqref{item:f_small_h}.
	
To show that \eqref{item:f_small_bound} holds for $f$, note that for every $z\in\C$,
\begin{equation}\label{eq_global_bound}
\abs{\chi(z)\cdot h(z)}\le z_1\exp\bb{\abs z^{\frac12+\eps}}
\end{equation}
by the way the sequence $\bset{z_n}$ was defined in \eqref{eq_fast_zn}. Then, using Cauchy's integral formula, Cauchy–Schwarz inequality, \eqref{eq_global_bound}, \hyperref[thm:Hormander]{Hörmander's Theorem},  property \ref{item_u_pasting_1} of the subharmonic function $u$ and \eqref{eq_C}, the following holds. 
\begin{align*}
		\abs{f(z)}&=\abs{\chi(z)\cdot h(z)-\alpha(z)}=\frac1\pi\abs{\integrate{B(z,1)}{}{\chi(w)\cdot h(w)-\alpha(w)}{m(w)}}\\
		& \le \frac1{\sqrt \pi}\sqrt{\integrate{B\bb{z, 1}}{}{\abs{\chi(w)\cdot h(w)-\alpha(w)}^2}{m(w)}}\\
		&\le 2 \abs {\chi(z)\cdot h(z)}+2\sqrt{\integrate{B(z,1)}{}{\abs{\alpha(w)}^2}{m(w)}}\\
		&\le 2 z_1\exp\bb{\abs z^{\frac12+\eps}}+2\underset{w\in B(z,1)}\max\; e^{\frac12u(w)}\cdot\bb{1+\abs w^2}\cdot\sqrt{ \integrate{B(z,1)}{}{\abs{\alpha(w)}^2\frac{e^{-u(w)}}{\bb{1+\abs w^2}^2}}{m(w)}}\\
		&\le 2 z_1\exp\bb{\abs z^{\frac12+\eps}}+2\underset{w\in B(z,1)}\max\; e^{\frac12\abs w^{\frac12+\eps}}\bb{1+\abs w^2}\cdot 10\cdot \frac{10\tilde C z_1}{\eps} \le \frac {C\cdot z_1}{\eps}\exp\bb{\abs z^{\frac12+\eps}}.
	\end{align*}
	
It is left to prove that the order of growth of $f$ is exactly $\frac12+\eps$. To bound the growth rate from above, note that using \eqref{item:f_small_bound}, for every $z\in\C$,
\begin{align*}
	\frac{\log\log\abs{f(z)}}{\log\abs z}\le \frac{\log\log\bb{\frac {C\cdot z_1}{\eps}\exp\bb{\abs z^{\frac12+\eps}}}}{\log\abs z}&< \frac{\log\bb{\frac {C\cdot z_1}{\eps}}+\log\bb{\abs z^{\frac12+\eps}}}{\log\abs z}\\
	&\le \frac12+\eps+\frac{\log\bb{\frac {C\cdot z_1}{\eps}}}{\log\abs z}.
\end{align*}
In particular, $\rho(f)\leq \frac 12+\eps$. To see a bound from below note that for every $z\in B\bb{z_n,\frac14}$ we have
\begin{align*}
	\frac{\log\log\abs{f(z)}}{\log\abs z}&\ge \frac{\log\log\bb{z_{n+1}-\abs{f(z)-h(z)}}}{\log\bb{z_n+\frac14}}\ge\frac{\log\log\bb{z_{n+1}-\frac {C\cdot z_1\abs {z}^2 }\eps\cdot e^{-\frac\eps C\abs {z}^{\frac12+\eps}}}}{\log\bb{z_n+\frac14}}\\
	&\ge \frac{\log\log\bb{z_{n+1}}}{\log\bb{z_n}}-o(1)=\frac{\log\log\left(\frac{\exp\bb{\frac16 z_n^{\frac12+\eps}}}n\right)}{\log(z_n)}-o(1)\ge \bb{\frac12+\eps}-o(1),
\end{align*}
as $n\rightarrow\infty$ implying that $\rho(f)\ge \frac 12+\eps$ concluding the proof.
\end{proof}
\bibliographystyle{alpha}
\bibliography{biblioUnbbWD}
\end{document}